\documentclass[
final
, nomarks
]{dmtcs-episciences}
\pdfoutput=1
\usepackage[utf8]{inputenc}
\usepackage{subfigure}

\usepackage{array}
\usepackage{mathtools}
\usepackage{multirow}
\usepackage{amsmath}
\usepackage{amsthm}
\usepackage{amssymb}

\makeatletter

\providecommand{\tabularnewline}{\\}

\numberwithin{equation}{section}
\theoremstyle{plain}
\newtheorem{thm}{\protect\theoremname}[section]
\theoremstyle{plain}
\newtheorem{lem}[thm]{\protect\lemmaname}
\theoremstyle{plain}
\newtheorem{prop}[thm]{\protect\propositionname}
\theoremstyle{plain}
\newtheorem{conjecture}[thm]{\protect\conjecturename}

\usepackage{bm}
\usepackage{tikz}
\usetikzlibrary{arrows}
\usepackage{hyperref}
\usepackage{colortbl}
\usepackage{xparse}
\usepackage{caption}
\usepackage{microtype}

\DeclareMathOperator{\pk}{pk}

\DeclareMathOperator{\Comp}{Comp}
\DeclareMathOperator{\des}{des}
\DeclareMathOperator{\st}{st}
\DeclareMathOperator{\std}{std}

\DeclareMathOperator{\val}{val}
\DeclareMathOperator{\fix}{fix}
\DeclareMathOperator{\pix}{pix}
\DeclareMathOperator{\rval}{rval}
\DeclareMathOperator{\dasc}{dasc}
\DeclareMathOperator{\ddes}{ddes}

\usepackage{titlesec}
\titleformat{\section}
  {\normalfont\large\bfseries}{\thesection.}{.8ex}{} 
\titleformat{\subsection}
  {\normalfont\bfseries}{\thesubsection.}{.8ex}{} 

\usetikzlibrary{shapes}

\tikzstyle{pathdefault}=[draw, line width=1, solid, color=black]
\tikzstyle{nodedefault}=[circle, inner sep=1.1, fill=black]
\tikzstyle{nodered}=[circle, inner sep=1.1, fill=red]
\tikzstyle{nodeblue}=[circle, inner sep=1.1, fill=blue]
\tikzstyle{empty}=[]
\tikzstyle{nodeellipsis}=[circle, inner sep=0.5, fill=black]
\tikzstyle{pathcolor1}=[draw, line width=1, solid, color=red]
\tikzstyle{pathcolor2}=[draw, line width=1, solid, color=blue]
\tikzstyle{pathcolorlight}=[draw, line width=1, dotted, color=lightgray]

\tikzstyle{arbpathcolor0}=[line width=1, dashdotted, color=black]
\tikzstyle{arbpathcolor1}=[line width=1, densely dashed, color=red]
\tikzstyle{arbpathdefault}=[line width=1, densely dotted, color=blue]

\newcounter{id}
\newcommand{\drawlinedotswithstyle}[4]{
 \def\x{{#3}}
 \def\y{{#4}}
 \tikzstyle{thispathstyle}=[#1]
 \tikzstyle{thisnodestyle}=[#2]
 \setcounter{id}{-1} 
 \foreach \j in {#3}{\stepcounter{id}} 
 \foreach \i in {1,...,\the\value{id}}{  
  \path[thispathstyle] (\x[\i],\y[\i]) --(\x[\i-1],\y[\i-1]); 
 }
 \foreach \i in {1,...,\the\value{id}}{  
  \node[thisnodestyle] at (\x[\i],\y[\i]) {}; 
 }
 \node[thisnodestyle] at (\x[0],\y[0]) {}; 
}

\DeclareDocumentCommand{\drawlinedots}{ O{pathdefault} O{nodedefault} m m}{\drawlinedotswithstyle{#1}{#2}{#3}{#4}}
\DeclareDocumentCommand{\drawlinedotsred}{ O{pathcolor1} O{nodered} m m}{\drawlinedotswithstyle{#1}{#2}{#3}{#4}}
\DeclareDocumentCommand{\drawlinedotsblue}{ O{pathcolor2} O{nodeblue} m m}{\drawlinedotswithstyle{#1}{#2}{#3}{#4}}

\let\originalleft\left
\let\originalright\right
\renewcommand{\left}{\mathopen{}\mathclose\bgroup\originalleft}
\renewcommand{\right}{\aftergroup\egroup\originalright}

\newcommand{\leqnomode}{\tagsleft@true\let\veqno\@@leqno}
\newcommand{\reqnomode}{\tagsleft@false\let\veqno\@@eqno}
\reqnomode

\usepackage{babel}
\providecommand{\conjecturename}{Conjecture}
\providecommand{\lemmaname}{Lemma}
\providecommand{\propositionname}{Proposition}
\providecommand{\theoremname}{Theorem}

\usepackage{titlesec}
\titleformat{\section}
  {\normalfont\large\bfseries}{\thesection.}{.8ex}{} 
\titleformat{\subsection}
  {\normalfont\bfseries}{\thesubsection.}{.8ex}{} 

\providecommand{\lemmaname}{Lemma}
\providecommand{\theoremname}{Theorem}
\providecommand{\propositionname}{Proposition}

\author[Chadi Bsila, Caroline E.\ Cox, Anna S.\ Hugo, Lindsey A.\ Styron, and Yan Zhuang]{Chadi Bsila \and Caroline E.\ Cox \and Anna S.\ Hugo \\ \and Lindsey A.\ Styron \and Yan Zhuang}
  
\title{Desarrangements revisited: statistics and pattern avoidance}
\affiliation{Department of Mathematics and Computer Science, Davidson College, Davidson, NC, USA}

\makeatother

\begin{document}

\publicationdata{vol. 27:1, Permutation Patterns 2024}{2025}{4}{10.46298/dmtcs.14375}{2024-10-01; 2024-10-01; 2025-07-29; 2025-09-06}{2025-10-13}

\maketitle

\begin{abstract}
\vspace{8bp}
A desarrangement is a permutation whose first ascent is even. Desarrangements
were introduced in the 1980s by Jacques D\'{e}sarm\'{e}nien, who
proved that they are in bijection with derangements. We revisit the
study of desarrangements, focusing on two themes: the refined enumeration
of desarrangements with respect to permutation statistics, and pattern
avoidance in desarrangements. Our main results include generating
function formulas for counting desarrangements by the number of descents,
peaks, valleys, double ascents, and double descents, as well as a
complete enumeration of desarrangements avoiding a prescribed set
of length 3 patterns. We find new interpretations of the Catalan,
Fine, Jacobsthal, and Fibonacci numbers in terms of pattern-avoiding
desarrangements.
\end{abstract}

{\let\thefootnote\relax}

\section{Introduction}

Let $\mathfrak{S}_{n}$ denote the symmetric group of permutations
of the set $[n]\coloneqq\{1,2,\dots,n\}$. We will exclusively write
permutations in one-line notation\textemdash that is, $\pi\in\mathfrak{S}_{n}$
is written as $\pi=\pi_{1}\pi_{2}\cdots\pi_{n}$. By convention, we
let $\mathfrak{S}_{0}$ consist of the empty permutation. The \textit{length}
$\left|\pi\right|$ of a permutation $\pi$ is its number of letters,
so if $\pi\in\mathfrak{S}_{n}$ then $\left|\pi\right|=n$.

Given a permutation $\pi$ of length $n$, we call $i\in[n-1]$ a
\textit{descent} of $\pi$ if $\pi_{i}>\pi_{i+1}$, and we let $\des(\pi)$
denote the number of descents of $\pi$. For example, the descents
of $\pi=31254$ are $1$ and $4$, so $\des(\pi)=2$. The descent
number $\des$ is a classical permutation statistic whose study dates
back to the work of MacMahon \cite{macmahon}. The distribution of
$\des$ over $\mathfrak{S}_{n}$ is given by the $n$th \textit{Eulerian
polynomial} 
\[
A_{n}(t)\coloneqq\sum_{\pi\in\mathfrak{S}_{n}}t^{\des(\pi)},
\]
and the sequence of Eulerian polynomials has the well-known exponential
generating function 
\begin{equation}
A(t,x)\coloneqq\sum_{n=0}^{\infty}A_{n}(t)\frac{x^{n}}{n!}=\frac{1-t}{e^{(t-1)x}-t}.\label{e-Euleriangf}
\end{equation}

Let us call $i\in[n]$ an \textit{ascent} of a length $n$ permutation
$\pi$ if $i$ is not a descent of $\pi$. Note that our definition
of ascent deviates from the definition that is most common in the
literature, which does not allow the last position of a permutation
to be an ascent.

A permutation is called a \textit{desarrangement} if its first ascent
is even. (By convention, we consider the empty permutation to be a
desarrangement as well.) Let $\mathfrak{D}_{n}$ denote the set of
desarrangements in $\mathfrak{S}_{n}$. See Table \ref{tb-desar}
for all (nonempty) desarrangements of length up to $n=5$. Notice
that the decreasing permutation $n\cdots21$ is a desarrangement if
and only if $n$ is even; we point this out because, when proving
facts about desarrangements, the decreasing desarrangements often
present a special case that must be considered separately.\renewcommand{\arraystretch}{1.2}
\begin{table}
\begin{centering}
\begin{tabular}{|c|c|c|c|c|c|ccccc|}
\cline{1-1} \cline{3-3} \cline{5-5} \cline{7-11} \cline{8-11} \cline{9-11} \cline{10-11} \cline{11-11} 
$\mathfrak{D}_{1}$ &  & $\mathfrak{D}_{2}$ &  & $\mathfrak{D}_{3}$ &  &  &  & $\mathfrak{D}_{4}$ &  & \tabularnewline
\cline{1-1} \cline{3-3} \cline{5-5} \cline{7-11} \cline{8-11} \cline{9-11} \cline{10-11} \cline{11-11} 
(none) & \quad{} & $21$ & \quad{} & $213$ & \quad{} & $2134$ & $3124$ & $3241$ & $4132$ & $4321$\tabularnewline
\cline{1-1} \cline{3-3} 
\multicolumn{1}{c}{} & \multicolumn{1}{c}{} & \multicolumn{1}{c}{} &  & $312$ &  & $2143$ & $3142$ & $4123$ & $4231$ & \tabularnewline
\cline{5-5} \cline{7-11} \cline{8-11} \cline{9-11} \cline{10-11} \cline{11-11} 
\end{tabular}\bigskip{}
\par\end{centering}
\begin{centering}
\begin{tabular}{|ccccccccccc|}
\hline 
 &  &  &  &  & $\mathfrak{D}_{5}$ &  &  &  &  & \tabularnewline
\hline 
$21345$ & $21534$ & $31425$ & $32415$ & $41235$ & $41523$ & $42513$ & $43521$ & $51342$ & $52341$ & $53412$\tabularnewline
$21354$ & $21543$ & $31452$ & $32451$ & $41253$ & $41532$ & $42531$ & $51234$ & $51423$ & $52413$ & $53421$\tabularnewline
$21435$ & $31245$ & $31524$ & $32514$ & $41325$ & $42315$ & $43215$ & $51243$ & $51432$ & $52431$ & $54213$\tabularnewline
$21453$ & $31254$ & $31542$ & $32541$ & $41352$ & $42351$ & $43512$ & $51324$ & $52314$ & $53214$ & $54312$\tabularnewline
\hline 
\end{tabular}
\par\end{centering}
\caption{\label{tb-desar}Desarrangements of length $\protect\leq5$.}
\end{table}

Let $d_{n}\coloneqq\left|\mathfrak{D}_{n}\right|$. The numbers $d_{n}$
are commonly called \textit{derangement numbers} because they also
count \textit{derangements}: permutations with no fixed points. The
first several derangement numbers are displayed below:
\begin{center}
\begin{tabular}{c|c|c|c|c|c|c|c|c|c|c|c|c}
$n$ & $0$ & $1$ & $2$ & $3$ & $4$ & $5$ & $6$ & $7$ & $8$ & $9$ & $10$ & $11$\tabularnewline
\hline 
$d_{n}$ & $1$ & $0$ & $1$ & $2$ & $9$ & $44$ & $265$ & $1854$ & $14833$ & $133496$ & $1334961$ & $14684570$\tabularnewline
\end{tabular}
\par\end{center}

\noindent See entry A000166 of the On-Line Encyclopedia of Integer
Sequences (OEIS) \cite{oeis}.

Desarrangements were introduced by Jacques D\'{e}sarm\'{e}nien for
the purpose of studying derangements. D\'{e}sarm\'{e}nien \cite{Desarmenien1984}
proved that desarrangements are in bijection with derangements, and
used desarrangements to give a combinatorial proof of the recurrence
$d_{n}=nd_{n-1}+(-1)^{n}$ for the derangement numbers. Shortly thereafter,
D\'{e}sarm\'{e}nien and Wachs \cite{Desarmenien1988,Desarmenien1993}
established a remarkable equidistribution concerning derangements
and desarrangements: the number of derangements with a given descent
set is equal to the the number of desarrangements whose inverses have
the same descent set. Foata and Han \cite{Foata2008a,Foata2008} later
introduced and studied a permutation statistic $\pix$, the number
of ``pixed points'', which we will define in Section \ref{ss-pkval}.
Just as derangements are permutations without fixed points, desarrangements
are precisely the permutations with no pixed points. See also \cite{Burstein2015,Foata2008b,Gessel1991,Han2009,Lin2013,Lin2016}
for a selection of other works related to desarrangements and pixed
points.

Nearly all of the existing literature on desarrangements focus on
their connection to derangements, and the work in this paper was motivated
by the desire to further study desarrangements for their own sake.
Here, we focus on two themes: the refined enumeration of desarrangements
by permutation statistics, and pattern avoidance in desarrangements.

Given a (non-negative integer-valued) permutation statistic $\st$,
let 
\[
D_{n}^{\st}(t)\coloneqq\sum_{\pi\in\mathfrak{D}_{n}}t^{\st(\pi)}
\]
be the polynomial which encodes the distribution of $\st$ over $\mathfrak{D}_{n}$,
and let 
\[
D^{\st}(t,x)\coloneqq\sum_{n=0}^{\infty}D_{n}^{\st}(t)\frac{x^{n}}{n!}
\]
denote the exponential generating function of the polynomials $D_{n}^{\st}(t)$.
The polynomials $D_{n}^{\des}(t)$ in particular capture the Eulerian
distribution over desarrangements, and we also consider four additional
statistics defined below. Given a permutation $\pi$ of length $n$,
we say that $i\in\{2,3,\dots,n-1\}$ is:
\begin{itemize}
\item a \textit{peak} of $\pi$ if $\pi_{i-1}<\pi_{i}>\pi_{i+1}$;
\item a \textit{valley} of $\pi$ if $\pi_{i-1}>\pi_{i}<\pi_{i+1}$;
\item a \textit{double ascent} of $\pi$ if $\pi_{i-1}<\pi_{i}<\pi_{i+1}$;
\item and a \textit{double descent} of $\pi$ if $\pi_{i-1}>\pi_{i}>\pi_{i+1}$.
\end{itemize}
Let $\pk(\pi)$, $\val(\pi)$, $\dasc(\pi)$, and $\ddes(\pi)$ denote
the number of peaks, valleys, double ascents, and double descents
of $\pi$, respectively. For example, if $\pi=214573689$, then $\pk(\pi)=1$,
$\val(\pi)=2$, $\dasc(\pi)=4$, and $\ddes(\pi)=0$. In Section \ref{s-stat},
we obtain formulas for the generating functions $D^{\des}(t,x)$,
$D^{\pk}(t,x)$, $D^{\val}(t,x)$, $D^{\dasc}(t,x)$, and $D^{\ddes}(t,x)$.
We will be able to apply (a generalization of) Gessel's ``run theorem''
\cite{gessel-thesis,Zhuang2016} to systematically derive our formulas,
although we will first prove some of them using more elementary means.

The study of pattern avoidance in permutations (and other combinatorial
structures) has received widespread attention in enumerative and algebraic
combinatorics, probability, and theoretical computer science since
the seminal work of Simion and Schmidt \cite{Simion1985}. To define
pattern avoidance, let us first introduce the notation $\std(w)$
for the \textit{standardization} of a word $w$. That is, if $w$
is a word consisting of $n$ distinct positive integers, then $\std(w)$
is the permutation in $\mathfrak{S}_{n}$ obtained by replacing the
smallest letter of $w$ by 1, the second smallest by 2, and so on.
For example, we have $\std(36815)=24513$. Given permutations $\sigma$
and $\pi$, an \textit{occurrence} of $\sigma$ in $\pi$ is a subsequence
of $\pi$ whose standardization is $\sigma$, and we say that $\pi$
\textit{avoids} $\sigma$ (as a \textit{pattern})\textemdash and that
$\pi$ is \textit{$\sigma$-avoiding}\textemdash if $\pi$ contains
no occurrences of $\sigma$. For example, $534$ is an occurrence
of $\sigma=312$ in $\pi=215364$, but $\pi=432651$ avoids $\sigma=312$.

Let $\mathfrak{S}_{n}(\sigma)$ denote the set of permutations in
$\mathfrak{S}_{n}$ which avoid $\sigma$. As noted above, we have
$432651\in\mathfrak{S}_{6}(312$). More generally, given a set $\Pi$
of patterns, let $\mathfrak{S}_{n}(\Pi)$ denote the set of permutations
in $\mathfrak{S}_{n}$ which avoid every pattern in $\Pi$. (Given
a specific $\Pi$, we will often omit the curly braces enclosing the
elements of $\Pi$ when writing out $\mathfrak{S}_{n}(\Pi)$ and in
similar notations.) For single patterns of length 3, it is well known
that for all $\sigma\in\mathfrak{S}_{3}$ and all $n\geq0$, the number
of permutations in $\mathfrak{S}_{n}(\sigma)$ is the $n$th Catalan
number. The enumeration of the avoidance classes $\mathfrak{S}_{n}(\Pi)$
for all length 3 pattern sets $\Pi\subseteq\mathfrak{S}_{3}$ was
completed by Simion and Schmidt \cite{Simion1985}. 

Robertson\textendash Saracino\textendash Zeilberger \cite{Robertson2002}
and Mansour\textendash Robertson \cite{Mansour2002} later studied
the distribution of $\fix$\textemdash the number of fixed points\textemdash over
all $\mathfrak{S}_{n}(\Pi)$ with $\Pi\subseteq\mathfrak{S}_{3}$
and $\left|\Pi\right|\leq3$. These results specialize to results
about pattern-avoiding derangements. We conduct an analogous study
of pattern-avoiding desarrangements.

Let $\mathfrak{D}_{n}(\Pi)$ denote the set of desarrangements in
$\mathfrak{D}_{n}$ which avoid every pattern in $\Pi$, and let $d_{n}(\Pi)\coloneqq\left|\mathfrak{D}_{n}(\Pi)\right|$.
In Section \ref{s-pa}, we will determine $d_{n}(\Pi)$ for all $\Pi\subseteq\mathfrak{S}_{3}$
\`{a} la Simion and Schmidt. Our results in this section are proved
using a mixture of combinatorial and generating function techniques,
and include new interpretations for the Catalan, Fine, Jacobsthal,
and Fibonacci numbers. We end Section \ref{s-pa} by comparing our
results on pattern-avoiding desarrangements with established results
on pattern-avoiding derangements, and presenting a list of conjectured
equidistributions of the statistics $\pix$ and $\fix$ over pattern
avoidance classes.

\section{\label{s-stat}Counting desarrangements by statistics}

Our goal in this section is to derive a formula for the generating
function 
\[
D^{\st}(t,x)=\sum_{n=0}^{\infty}D_{n}^{\st}(t)\frac{x^{n}}{n!}=\sum_{n=0}^{\infty}\sum_{\pi\in\mathfrak{D}_{n}}t^{\st(\pi)}\frac{x^{n}}{n!}
\]
for each of the statistics $\des$, $\pk$, $\val$, $\dasc$, and
$\ddes$. We will complete this task for the first three statistics
using elementary techniques, before introducing (a generalization
of) Gessel's run theorem, which provides a method for systematically
deriving our formulas for all five statistics. To end this section,
we will provide two examples showing how the run theorem can be used
to study the joint distribution of multiple statistics.

\subsection{Descents}

Our first result, for the exponential generating function $D^{\des}(t,x)$,
shall be proven by deriving and solving an appropriate differential
equation.
\begin{thm}
\label{t-des}We have
\[
D^{\des}(t,x)=\frac{(1-t)(1-2t-e^{-tx}+e^{(t-1)x})}{(1-2t)(e^{(t-1)x}-t)}.
\]
\end{thm}

See Table \ref{tb-des} for the first ten polynomials $D_{n}^{\des}(t)$.

\renewcommand{\arraystretch}{1.2}
\begin{table}
\begin{centering}
\begin{tabular}{|c|c|c|c|c|}
\cline{1-2} \cline{2-2} \cline{4-5} \cline{5-5} 
$n$ & $D_{n}^{\des}(t)$ & \quad{} & $n$ & $D_{n}^{\des}(t)$\tabularnewline
\cline{1-2} \cline{2-2} \cline{4-5} \cline{5-5} 
$0$ & $1$ &  & $5$ & $4t+27t^{2}+13t^{3}$\tabularnewline
\cline{1-2} \cline{2-2} \cline{4-5} \cline{5-5} 
$1$ & $0$ &  & $6$ & $5t+94t^{2}+137t^{3}+28t^{4}+t^{5}$\tabularnewline
\cline{1-2} \cline{2-2} \cline{4-5} \cline{5-5} 
$2$ & $t$ &  & $7$ & $6t+270t^{2}+952t^{3}+566t^{4}+60t^{5}$\tabularnewline
\cline{1-2} \cline{2-2} \cline{4-5} \cline{5-5} 
$3$ & $2t$ &  & $8$ & $7t+699t^{2}+5093t^{3}+6825t^{4}+2085t^{5}+123t^{6}+t^{7}$\tabularnewline
\cline{1-2} \cline{2-2} \cline{4-5} \cline{5-5} 
$4$ & $3t+5t^{2}+t^{3}$ &  & $9$ & $8t+1701t^{2}+23195t^{3}+60513t^{4}+40649t^{5}+7179t^{6}+251t^{7}$\tabularnewline
\cline{1-2} \cline{2-2} \cline{4-5} \cline{5-5} 
\end{tabular}
\par\end{centering}
\caption{\label{tb-des}Distribution of $\protect\des$ over $\mathfrak{D}_{n}$
for $0\protect\leq n\protect\leq9$.}
\end{table}

\begin{proof}
For convenience, let $D^{\des}\coloneqq D^{\des}(t,x)$ and let $A\coloneqq A(t,x)$;
see (\ref{e-Euleriangf}). We first prove that the differential equation
\begin{equation}
\frac{\partial D^{\des}}{\partial x}=D^{\des}-1+t(D^{\des}-1)(A-1)+t(A-D^{\des})\label{e-desde}
\end{equation}
holds. Any nonempty desarrangement $\pi\in\mathfrak{D}_{n}$ can be
written as $\pi=\tau n\rho$, so that $\tau$ is the subsequence of
letters to the left of $n$ and $\rho$ is the subsequence of letters
to its right. If $\rho$ is empty, then $\tau$ is a nonempty desarrangement
with the same number of descents as $\pi$; this case contributes
the term $D^{\des}-1$ to (\ref{e-desde}). If $\tau$ is empty, then
$\rho$ is a non-desarrangement with one less descent than $\pi$;
this case contributes $t(A-D^{\des})$ to (\ref{e-desde}). Finally,
if neither $\tau$ nor $\rho$ are empty, then $\des(\pi)=\des(\tau)+\des(\rho)+1$;
this case contributes the remaining term $t(D^{\des}-1)(A-1)$.

Let us now rewrite (\ref{e-desde}) as
\begin{align}
\frac{\partial D^{\des}}{\partial x} & =(1-2t+tA)D^{\des}+t-1\nonumber \\
 & =\left(1-2t+\frac{t(1-t)}{e^{(t-1)x}-t}\right)D^{\des}+t-1,\label{e-desde2}
\end{align}
where we obtain the second line by substituting in (\ref{e-Euleriangf}).
We shall solve (\ref{e-desde2}) using the method of variation of
parameters. The homogeneous linear differential equation 
\[
\frac{\partial D_{0}}{\partial x}=\left(1-2t+\frac{t(1-t)}{e^{(t-1)x}-t}\right)D_{0}
\]
has general solution 
\[
D_{0}=\frac{K(t)e^{-tx}}{e^{(t-1)x}-t}
\]
where $K(t)$ is constant with respect to $x$. Since
\[
(t-1)e^{tx}\left(\frac{e^{(t-1)x}}{2t-1}-1\right)
\]
is an antiderivative of $e^{-tx}/(e^{(t-1)x}-t)$, it follows that
(\ref{e-desde2}) has general solution 
\begin{equation}
D^{\des}(x,t)=\frac{t-1}{e^{(t-1)x}-t}\left(\frac{e^{(t-1)x}}{2t-1}-1\right)+\frac{K(t)e^{-tx}}{e^{(t-1)x}-t}.\label{e-desgensol}
\end{equation}
We now use the initial condition $D^{\des}(0,t)=1$ to deduce $K(t)=(t-1)/(1-2t)$
from (\ref{e-desgensol}), and substituting this expression for $K(t)$
into (\ref{e-desgensol}) yields the desired formula.
\end{proof}

\subsection{\label{ss-pkval}Peaks and valleys}

We now proceed to the peak and valley number statistics. It is easy
to see that $\pk$ and $\val$ are equidistributed over the full symmetric
group $\mathfrak{S}_{n}$. Indeed, consider the\textit{ complement}
$\pi^{c}$ of $\pi\in\mathfrak{S}_{n}$ defined by 
\[
\pi^{c}\coloneqq(n+1-\pi_{1})\,(n+1-\pi_{2})\,\cdots\,(n+1-\pi_{n});
\]
then $\pk(\pi)=\val(\pi^{c})$ for all permutations $\pi$. However,
the property of being a desarrangement is not invariant under complementation,
and the distributions of these two statistics over $\mathfrak{D}_{n}$
are different. For example, the defining property of a desarrangement
guarantees that every nondecreasing desarrangement has at least one
valley, while there are far more desarrangements with no peaks. The
first ten polynomials $D^{\pk}(t,x)$ and $D^{\val}(t,x)$ are displayed
in Tables \ref{tb-pk} and \ref{tb-val}, respectively, and the next
theorem gives their exponential generating functions.
\begin{table}
\begin{centering}
\begin{tabular}{|c|c|c|c|c|}
\cline{1-2} \cline{2-2} \cline{4-5} \cline{5-5} 
$n$ & $D_{n}^{\pk}(t)$ & \quad{} & $n$ & $D_{n}^{\pk}(t)$\tabularnewline
\cline{1-2} \cline{2-2} \cline{4-5} \cline{5-5} 
$0$ & $1$ &  & $5$ & $8+36t$\tabularnewline
\cline{1-2} \cline{2-2} \cline{4-5} \cline{5-5} 
$1$ & $0$ &  & $6$ & $16+188t+61t^{2}$\tabularnewline
\cline{1-2} \cline{2-2} \cline{4-5} \cline{5-5} 
$2$ & $1$ &  & $7$ & $32+864t+958t^{2}$\tabularnewline
\cline{1-2} \cline{2-2} \cline{4-5} \cline{5-5} 
$3$ & $2$ &  & $8$ & $64+3728t+9656t^{2}+1385t^{3}$\tabularnewline
\cline{1-2} \cline{2-2} \cline{4-5} \cline{5-5} 
$4$ & $4+5t$ &  & $9$ & $128+15552t+79760t^{2}+38056t^{3}$\tabularnewline
\cline{1-2} \cline{2-2} \cline{4-5} \cline{5-5} 
\end{tabular}
\par\end{centering}
\caption{\label{tb-pk}Distribution of $\protect\pk$ over $\mathfrak{D}_{n}$
for $0\protect\leq n\protect\leq9$.}
\end{table}
\begin{table}
\begin{centering}
\begin{tabular}{|c|c|c|c|c|}
\cline{1-2} \cline{2-2} \cline{4-5} \cline{5-5} 
$n$ & $D_{n}^{\val}(t)$ & \quad{} & $n$ & $D_{n}^{\val}(t)$\tabularnewline
\cline{1-2} \cline{2-2} \cline{4-5} \cline{5-5} 
$0$ & $1$ &  & $5$ & $28t+16t^{2}$\tabularnewline
\cline{1-2} \cline{2-2} \cline{4-5} \cline{5-5} 
$1$ & $0$ &  & $6$ & $1+88t+176t^{2}$\tabularnewline
\cline{1-2} \cline{2-2} \cline{4-5} \cline{5-5} 
$2$ & $1$ &  & $7$ & $270t+1312t^{2}+272t^{3}$\tabularnewline
\cline{1-2} \cline{2-2} \cline{4-5} \cline{5-5} 
$3$ & $2t$ &  & $8$ & $1+816t+8256t^{2}+5760t^{3}$\tabularnewline
\cline{1-2} \cline{2-2} \cline{4-5} \cline{5-5} 
$4$ & $1+8t$ &  & $9$ & $2456t+47520t^{2}+75584t^{3}+7936t^{4}$\tabularnewline
\cline{1-2} \cline{2-2} \cline{4-5} \cline{5-5} 
\end{tabular}
\par\end{centering}
\caption{\label{tb-val}Distribution of $\protect\val$ over $\mathfrak{D}_{n}$
for $0\protect\leq n\protect\leq9$.}
\end{table}

\begin{thm}
\label{t-pkval}We have \leqnomode
\[
\tag{{a}}\qquad D^{\pk}(t,x)=1-t^{-1}+\frac{t^{-1}e^{-x}\sqrt{1-t}}{\sqrt{1-t}\cosh\left(x\sqrt{1-t}\right)-\sinh\left(x\sqrt{1-t}\right)}
\]
and 
\[
\tag{{b}}\qquad D^{\val}(t,x)=\frac{e^{-x}\sqrt{1-t}}{\sqrt{1-t}-\tanh(x\sqrt{1-t})}.
\]
\end{thm}

Let 
\[
P\coloneqq P(t,x)=\sum_{n=0}^{\infty}\sum_{\pi\in\mathfrak{S}_{n}}t^{\pk(\pi)}\frac{x^{n}}{n!}=\sum_{n=0}^{\infty}\sum_{\pi\in\mathfrak{S}_{n}}t^{\val(\pi)}\frac{x^{n}}{n!}.
\]
It is well known\textemdash see, e.g., \cite[Theorem 9]{Zhuang2016}\textemdash that
\begin{equation}
P(t,x)=\frac{1}{1-\frac{\tanh(x\sqrt{1-t})}{\sqrt{1-t}}}.\label{e-P}
\end{equation}
Similar reasoning as in the proof of Theorem \ref{t-des} gives us
the differential equations
\[
\frac{\partial D^{\pk}}{\partial x}=\left(1+t(D^{\pk}-1)\right)(P-1)
\]
and 
\[
\frac{\partial D^{\val}}{\partial x}=\left(D^{\val}-e^{-x}\right)\left(1+t(P-1)\right),
\]
which can be solved in the same way as our differential equation for
$D^{\des}$. However, let us also demonstrate a different proof which
does not involve differential equations.

Observe from Theorem \ref{t-pkval} (b) and Equation (\ref{e-P})
that $D^{\val}(t,x)=e^{-x}P(t,x)$, and there is a simple combinatorial
explanation for this fact. Every permutation $\pi$ can be written
uniquely as a concatenation $\pi=\iota\delta$ where $\iota$ is an
increasing sequence and $\delta$ is a desarrangement. For example,
$\pi=46785213$ is the concatenation of $\iota=467$ and $\delta=85213$.
Foata and Han \cite{Foata2008a} call this decomposition the \textit{pixed
factorization} of $\pi$ and the letters of $\iota$ \textit{pixed
points}. Since $\iota$ contributes no valleys to $\pi$, we have
\begin{equation}
P(t,x)=e^{x}D^{\val}(t,x),\label{e-PD}
\end{equation}
and now Theorem \ref{t-pkval} (b) follows immediately from substituting
(\ref{e-P}) into (\ref{e-PD}).

We note that Gessel's ``hook factorization'' \cite{Gessel1991} provides
a refinement of Foata and Han's pixed factorization, and we may equivalently
define pixed points to be the letters in the first factor of the hook
factorization. However, we will not need the hook factorization for
our present work.

One may hope that the pixed factorization can be used to prove results
analogous to (\ref{e-PD}) but for other statistics. However, this
can be done only if prepending an increasing sequence to a desarrangement
changes the statistic value in a manageable way, which is the case
for the $\val$ statistic (as the number of valleys is unchanged)
but is not true in general. For example, prepending an increasing
sequence $\iota$ to a desarrangement $\delta$ may either add one
to the number of peaks or leave it unchanged, and this depends on
whether the last letter of $\iota$ is greater than the first letter
of $\delta$ (along with the lengths of $\iota$ and $\delta$). Nonetheless,
we can prove our formula for $D^{\pk}(t,x)$ by considering another
statistic, closely related to valleys, which also behaves nicely with
respect to the pixed factorization.

\begin{proof}[of Theorem \ref{t-pkval} \textup{(}a\textup{)}]
Let us call $i\in\{2,3,\dots,n\}$ a \textit{right valley} of $\pi\in\mathfrak{S}_{n}$
if either $i$ is a valley of $\pi$, or if $i=n$ and $\pi_{n-1}>\pi_{n}$.
Let $\rval(\pi)$ denote the number of right valleys of $\pi$, and
let 
\[
\grave{P}(t,x)\coloneqq\sum_{n=0}^{\infty}\sum_{\pi\in\mathfrak{S}_{n}}t^{\rval(\pi)}\frac{x^{n}}{n!}.
\]
Since the pixed points of a permutation contribute no right valleys,
we have
\[
\grave{P}(t,x)=e^{x}D^{\rval}(t,x).
\]
Moreover, it is straightforward to verify that every nonempty desarrangement
$\pi$ satisfies $\rval(\pi)=\pk(\pi)+1$. Therefore, we have
\begin{align}
D^{\pk}(t,x) & =1+t^{-1}\left(D^{\rval}(t,x)-1\right)=1+t^{-1}\left(e^{-x}\grave{P}(t,x)-1\right).\label{e-pkrval}
\end{align}
Substituting the known formula
\[
\grave{P}(t,x)=\frac{\sqrt{1-t}}{\sqrt{1-t}\cosh(x\sqrt{1-t})-\sinh(x\sqrt{1-t})}
\]
\cite[Theorem 10]{Zhuang2016} into (\ref{e-pkrval}) gives us our
desired formula for $D^{\pk}(t,x)$.
\end{proof}

\subsection{The run theorem}

While we were able to obtain our first several results by elementary
means, let us now introduce a more sophisticated technique that can
be used to count desarrangements by a wide variety of statistics related
to the ``run structure'' of permutations.

Every permutation can be uniquely decomposed as a sequence of \textit{increasing
runs}: maximal increasing consecutive subsequences. For instance,
the increasing runs of $317542689$ are $3$, $17$, $5$, $4$, and
$2689$. We will often refer to increasing runs as simply \textit{runs}
for the sake of brevity. Given a permutation $\pi$, let $\Comp\pi$
denote the \textit{descent composition} of $\pi$: the integer composition
whose parts are the lengths of the increasing runs of $\pi$ in the
order that they appear. Continuing our example, we have $\Comp317542689=(1,2,1,1,4)$.
Let us call an increasing run \textit{short} if it has length 1, and
\textit{long} if it has length at least 2.

Various families of permutations are characterized by restrictions
on their run structure. For example, (up-down) alternating permutations
are precisely the permutations with descent compositions of the form
$(2,2,\dots,2)$ or $(2,2,\dots,2,1)$. Desarrangements are permutations
which either begin with an odd number of short runs before its first
long run, or consist only of an even number of short runs. Similarly,
all of the permutation statistics we have considered thus far are
expressible in terms of their run structure. For instance, the number
of descents of a permutation is one less than its number of increasing
runs, and its number of valleys is equal to its number of non-initial
long runs (every valley is at the beginning of a non-initial long
run, and every non-initial long run begins with a valley). 

Gessel's run theorem \cite{gessel-thesis}, along with a generalized
version due to Zhuang \cite{Zhuang2016}, provides a systematic method
for counting permutations with restrictions on runs by statistics
expressible in terms of runs. The run theorem is a reciprocity formula
involving noncommutative symmetric functions, and the generalized
version involves matrices of noncommutative symmetric functions, but
here we will provide a treatment of the (generalized) run theorem
without using the language of noncommutative symmetric functions.

Fix a positive integer $m$, and let $G$ be a directed graph on the
vertex set $[m]$ such that each edge $(i,j)$ of $G$ is assigned
a set $P_{i,j}$ of positive integers. Let us say that an integer
composition $L=(L_{1},L_{2},\dots,L_{k})$ is ($i,j)$-\textit{admissible}
if there exists in $G$ a sequence of edges $(i_{1},i_{2})(i_{2},i_{3})\cdots(i_{k},i_{k+1})$\textemdash where
$i=i_{1}$ and $j=i_{k+1}$\textemdash such that $L_{l}\in P_{i_{l},i_{l+1}}$
for all $1\leq l\leq k$; in this case, we say that $L$ is $(i,j)$-admissible
\textit{along} \textit{the path} $(i_{1},i_{2})(i_{2},i_{3})\cdots(i_{k},i_{k+1})$.

For example, suppose that $G$ is the directed graph in Figure \ref{f-desar}.
Then $(1,1,1,3,4,1,2)$ is $(1,3)$-admissible whereas $(2,1,1,4,3)$
is not. In fact, the $(1,3)$-admissible compositions are precisely
the descent compositions of nondecreasing desarrangements.
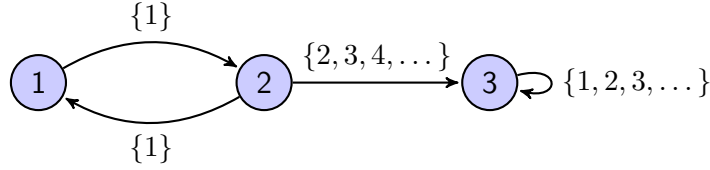
\begin{figure}
\begin{center}
\begin{tikzpicture}[->,>=stealth',shorten >=1pt,auto,node distance=3cm,   thick,main node/.style={circle,fill=blue!20,draw,font=\sffamily}]

\node[main node] (1) {1};
\node[main node] (2) [right of =1] {2};
\node[main node] (3) [right of =2] {3};

\path[every node/.style={font=\sffamily\small}]     
(1) edge [bend left] node {$\{1\}$} (2)
(2) edge [bend left] node {$\{1\}$} (1)
(2) edge node {$\{2,3,4,\dots\}$} (3)
(3) edge [loop right] node {$\{1,2,3,\dots\}$} (3);

\end{tikzpicture}
\end{center}

\caption{\label{f-desar}Directed graph for desarrangements.}
\end{figure}

Given a matrix $A$ consisting of ordinary generating functions in
the variable $x$, let $\hat{A}$ be the matrix obtained from $A$
by replacing each entry with its corresponding exponential generating
function; in other words, we apply to each entry the map defined by
$x^{n}\mapsto x^{n}/n!$ for all $n\geq0$ and extending linearly.
For example, 
\[
\text{if }A=\begin{bmatrix}1 & {\displaystyle \frac{x^{3}}{1-x^{2}}}\vphantom{{\displaystyle \frac{\frac{dy}{dx}}{\frac{dy}{dx}}}}\\
0 & x+x^{2}+x^{4}
\end{bmatrix},\quad\text{then }\hat{A}=\begin{bmatrix}1 & \sinh x-x\\
0 & x+{\displaystyle \frac{x^{2}}{2}}+{\displaystyle \frac{x^{4}}{24}}\vphantom{{\displaystyle \frac{\frac{dy}{dx}}{\frac{dy}{dx}}}}
\end{bmatrix}.
\]
The formula
\[
\hat{A}(x)=\frac{1}{2\pi}\int_{-\pi}^{\pi}A(xe^{-i\theta})e^{e^{i\theta}}d\theta
\]
\cite[p. 566]{Graham1994} can be used for this conversion. We shall
use the notation $\mathbf{I}_{m}$ for the $m$-by-$m$ identity matrix.
\begin{thm}[Generalized run theorem]
\label{t-runthm}Let $G$ be a directed graph on $[m]$ such that
each edge $(i,j)$ of $G$ is assigned a set $P_{i,j}$ of positive
integers. Suppose that, for all $i,j\in[m]$, no composition is $(i,j)$-admissible
along more than one path. Let $\{\,w_{a}^{(i,j)}:a\in P_{i,j}\,\}$
be a set of weights, with $w_{a}^{(i,j)}=0$ if $a\notin P_{i,j}$.
Given a composition $L=(L_{1},L_{2},\dots,L_{k})$ and $1\leq i,j\leq m$,
let $w^{(i,j)}(L)=w_{L_{1}}^{e_{1}}w_{L_{2}}^{e_{k}}\cdots w_{L_{k}}^{e_{k}}$
if $L$ is $(i,j)$-admissible along the path $e_{1}e_{2}\cdots e_{k}$,
and let $w^{(i,j)}(L)=0$ otherwise. Taking
\[
A\coloneqq\left(\mathbf{I}_{m}+\begin{bmatrix}{\displaystyle \sum_{k=1}^{\infty}w_{k}^{(1,1)}x^{k}}\vphantom{{\displaystyle \frac{\frac{dy_{a}}{dx_{a}}}{\frac{dy_{a}}{dx_{a}}}}} & \cdots & {\displaystyle \sum_{k=1}^{\infty}w_{k}^{(1,m)}x^{k}}\\
\vdots & \ddots & \vdots\\
{\displaystyle \sum_{k=1}^{\infty}w_{k}^{(m,1)}x^{k}}\vphantom{{\displaystyle \frac{\frac{dy_{a}}{dx_{a}}}{\frac{dy_{a}}{dx_{a}}}}} & \cdots & {\displaystyle \sum_{k=1}^{\infty}w_{k}^{(m,m)}x^{k}}
\end{bmatrix}\right)^{-1},
\]
the exponential generating function 
\[
\sum_{n=0}^{\infty}\sum_{\pi\in\mathfrak{S}_{n}}w^{(i,j)}(\Comp\pi)\frac{x^{n}}{n!}
\]
is the $(i,j)$th entry of $(\hat{A})^{-1}$.
\end{thm}

This theorem is the result of \cite[Theorem 2]{Zhuang2016} upon applying
the algebra homomorphism $\mathbf{h}_{n}\mapsto x^{n}/n!$ (where
$\mathbf{h}_{n}$ is the noncommutative complete symmetric function
of homogeneous degree $n$). As such, we can think of Theorem \ref{t-runthm}
as being the generalized run theorem projected down to the setting
of exponential generating functions, although in this paper we shall
refer to Theorem \ref{t-runthm} itself as the ``generalized run theorem''.
It is worth mentioning that variations of Theorem \ref{t-runthm}
can be obtained using other homomorphisms. For example, applying the
homomorphism $\mathbf{h}_{n}\mapsto x^{n}/[n]_{q}!$\textemdash where
$[n]_{q}!$ is the $n$th $q$-factorial\textemdash leads to a $q$-analogue
of Theorem \ref{t-runthm} which refines by the inversion number.
See \cite{Gessel2020} for an exposition of some of these homomorphisms.

As a preliminary example showcasing the utility of the generalized
run theorem, let us derive the well-known exponential generating function
$e^{-x}/(1-x)$ for the derangement numbers. Take $G$ to be the graph
in Figure \ref{f-desar} and set all weights $w_{k}^{(i,j)}$ equal
to 1. Observe that
\[
w^{(1,3)}(\Comp\pi)=\begin{cases}
1, & \text{if }\pi\text{ is a nondecreasing desarrangement,}\\
0, & \text{otherwise,}
\end{cases}
\]
so we will apply the generalized run theorem to find the exponential
generating function 
\[
\sum_{n=0}^{\infty}\sum_{\pi\in\mathfrak{S}_{n}}w^{(1,3)}(\Comp\pi)\frac{x^{n}}{n!}
\]
and then add the exponential generating function $\cosh x$ for decreasing
desarrangements. We begin with the matrix
\[
A=\begin{bmatrix}1 & x & 0\\
x & 1 & {\displaystyle \frac{x^{2}}{1-x}}\vphantom{{\displaystyle \frac{\frac{dy}{dx}}{\frac{dy}{dx}}}}\\
0 & 0 & {\displaystyle \frac{1}{1-x}}\vphantom{{\displaystyle \frac{\frac{dy}{dx}}{\frac{dy}{dx}}}}
\end{bmatrix}^{-1}=\begin{bmatrix}{\displaystyle \frac{1}{1-x^{2}}}\vphantom{{\displaystyle \frac{\frac{dy}{dx}}{\frac{dy}{dx}}}} & {\displaystyle \frac{-x}{1-x^{2}}} & {\displaystyle \frac{x^{3}}{1-x^{2}}}\\
{\displaystyle \frac{-x}{1-x^{2}}}\vphantom{{\displaystyle \frac{\frac{dy}{dx}}{\frac{dy}{dx}}}} & {\displaystyle \frac{1}{1-x^{2}}} & {\displaystyle \frac{-x^{2}}{1-x^{2}}}\\
0 & 0 & 1-x
\end{bmatrix}
\]
of ordinary generating functions, and the matrix of the corresponding
exponential generating functions is
\[
\hat{A}=\begin{bmatrix}\cosh x & {\displaystyle -\sinh x} & \sinh x-x\\
{\displaystyle -\sinh x} & \cosh x & {\displaystyle 1-\cosh x}\\
0 & 0 & 1-x
\end{bmatrix}.
\]
Taking the inverse of $\hat{A}$ yields 
\[
(\hat{A})^{-1}=\begin{bmatrix}\cosh x & {\displaystyle \sinh x} & {\displaystyle \frac{x\cosh x-\sinh x}{1-x}}\vphantom{{\displaystyle \frac{\frac{dy}{dx}}{\frac{dy}{dx}}}}\\
{\displaystyle \sinh x} & \cosh x & {\displaystyle \frac{1-\cosh x+x\sinh x}{1-x}}\vphantom{{\displaystyle \frac{\frac{dy}{dx}}{\frac{dy}{dx}}}}\\
0 & 0 & {\displaystyle \frac{1}{1-x}}\vphantom{{\displaystyle \frac{\frac{dy}{dx}}{\frac{dy}{dx}}}}
\end{bmatrix}.
\]
Adding $\cosh x$ to the $(1,3)$ entry of $(\hat{A})^{-1}$ and simplifying
yields $e^{-x}/(1-x)$, as expected.

The same generating function can be obtained by considering paths
from $1$ to $2$ in the directed graph in Figure \ref{f-cdesar},
which give the descent compositions of complements of desarrangements
(except for the empty permutation), and doing so is slightly easier
as we would only need to work with 2-by-2 matrices as opposed to 3-by-3
ones. 
\begin{figure}
\begin{center}
\begin{tikzpicture}[->,>=stealth',shorten >=1pt,auto,node distance=3cm,   thick,main node/.style={circle,fill=blue!20,draw,font=\sffamily}]

\node[main node] (1) {1};
\node[main node] (2) [right of =1] {2};

\path[every node/.style={font=\sffamily\small}]     
(1) edge node {$\{2,4,6,\dots\}$} (2)
(2) edge [loop right] node {$\{1,2,3,\dots\}$} (2);

\end{tikzpicture}
\end{center}

\caption{\label{f-cdesar}Directed graph for complements of desarrangements.}
\end{figure}
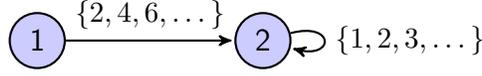

As a further example, let us rederive our formula in Theorem \ref{t-des}
for the generating function $D^{\des}(t,x)$. Notice that descents
of a permutation are precisely the (non-final) ascents of its complement,
and that each increasing run of length $k$ contributes $k-1$ such
ascents. So, let us consider paths from $1$ to $2$ in the directed
graph in Figure \ref{f-cdesar}, along with weights $w_{k}^{(1,2)}=t^{k-1}$
for each $k\in\{2,4,6,\dots\}$ and $w_{k}^{(2,2)}=t^{k-1}$ for each
$k\geq1$. Then 
\[
A=\begin{bmatrix}1 & {\displaystyle \frac{tx^{2}}{1-t^{2}x^{2}}}\vphantom{{\displaystyle \frac{\frac{dy}{dx}}{\frac{dy}{dx}}}}\\
0 & {\displaystyle 1+\frac{x}{1-tx}}\vphantom{{\displaystyle \frac{\frac{dy}{dx}}{\frac{dy}{dx}}}}
\end{bmatrix}^{-1}=\begin{bmatrix}1 & {\displaystyle \frac{-tx^{2}}{(1+tx)(1+x-tx)}}\vphantom{{\displaystyle \frac{\frac{dy}{dx}}{\frac{dy}{dx}}}}\\
0 & {\displaystyle \frac{1-tx}{1+x-tx}}\vphantom{{\displaystyle \frac{\frac{dy}{dx}}{\frac{dy}{dx}}}}
\end{bmatrix},
\]
and the corresponding matrix of exponential generating functions is
\[
\hat{A}=\begin{bmatrix}1 & {\displaystyle \frac{(1-t)e^{-tx}-te^{(t-1)x}+2t-1}{(1-t)(1-2t)}}\vphantom{{\displaystyle \frac{\frac{dy}{dx}}{\frac{dy}{dx}}}}\\
0 & {\displaystyle \frac{e^{(t-1)x}-t}{1-t}}\vphantom{{\displaystyle \frac{\frac{dy}{dx}}{\frac{dy}{dx}}}}
\end{bmatrix}.
\]
The inverse of $\hat{A}$ is 
\[
(\hat{A})^{-1}=\begin{bmatrix}1 & {\displaystyle \frac{1-2t-(1-t)e^{-tx}+te^{(t-1)x}}{(1-2t)(e^{(t-1)x}-t)}}\vphantom{{\displaystyle \frac{\frac{dy}{dx}}{\frac{dy}{dx}}}}\\
0 & {\displaystyle \frac{1-t}{e^{(t-1)x}-t}}\vphantom{{\displaystyle \frac{\frac{dy}{dx}}{\frac{dy}{dx}}}}
\end{bmatrix},
\]
and adding $1$ to the $(1,2)$ entry of $(\hat{A})^{-1}$ to account
for the empty permutation recovers Theorem \ref{t-des}.

The generalized run theorem can also be used to recover our formulas
for $D^{\pk}(t,x)$ and $D^{\val}(t,x)$ in Theorem \ref{t-pkval}.
In fact, in Section \ref{ss-joint}, we'll use the generalized run
theorem to prove a formula for the joint distribution of $\pk$ and
$\des$ over $\mathfrak{D}_{n}$.

\subsection{Double ascents and double descents}

Let us now use the generalized run theorem to derive the following
formulas for the generating functions $D^{\dasc}(t,x)$ and $D^{\ddes}(t,x)$.
\begin{thm}
\label{t-dascddes}Let $\alpha=\sqrt{(t+3)(t-1)}$. We have \leqnomode
\[
\tag{{a}}\qquad D^{\dasc}(t,x)=\frac{e^{-x}\left((2-t)\alpha\cosh(\frac{\alpha x}{2})-t(1-t)\sinh(\frac{\alpha x}{2})\right)+(1-t)\alpha e^{(1-t)x/2}}{(3-2t)\left(\alpha\cosh\left(\frac{\alpha x}{2}\right)-(1+t)\sinh\left(\frac{\alpha x}{2}\right)\right)}
\]
and 
\[
\tag{{b}}\qquad D^{\ddes}(t,x)=\frac{2t(1-t)\alpha\cosh\left(\frac{\alpha x}{2}\right)+2(1-2t+t^{3})\sinh\left(\frac{\alpha x}{2}\right)-\alpha e^{(1-3t)x/2}}{(2t(1-t)-1)\left(\alpha\cosh\left(\frac{\alpha x}{2}\right)-(1+t)\sinh\left(\frac{\alpha x}{2}\right)\right)}.
\]
\end{thm}

See Tables \ref{tb-dasc}\textendash \ref{tb-ddes} for the first
ten polynomials $D_{n}^{\dasc}(t)$ and $D_{n}^{\ddes}(t)$, respectively.
\begin{table}
\begin{centering}
\begin{tabular}{|c|c|c|c|c|}
\cline{1-2} \cline{2-2} \cline{4-5} \cline{5-5} 
$n$ & $D_{n}^{\dasc}(t)$ & \quad{} & $n$ & $D_{n}^{\dasc}(t)$\tabularnewline
\cline{1-2} \cline{2-2} \cline{4-5} \cline{5-5} 
$0$ & $1$ &  & $5$ & $29+11t+4t^{2}$\tabularnewline
\cline{1-2} \cline{2-2} \cline{4-5} \cline{5-5} 
$1$ & $0$ &  & $6$ & $130+111t+19t^{2}+5t^{3}$\tabularnewline
\cline{1-2} \cline{2-2} \cline{4-5} \cline{5-5} 
$2$ & $1$ &  & $7$ & $798+705t+316t^{2}+29t^{3}+6t^{4}$\tabularnewline
\cline{1-2} \cline{2-2} \cline{4-5} \cline{5-5} 
$3$ & $2$ &  & $8$ & $5125+6242t+2626t^{2}+792t^{3}+41t^{4}+7t^{5}$\tabularnewline
\cline{1-2} \cline{2-2} \cline{4-5} \cline{5-5} 
$4$ & $6+3t$ &  & $9$ & $38726+52830t+31794t^{2}+8220t^{3}+1863t^{4}+55t^{5}+8t^{6}$\tabularnewline
\cline{1-2} \cline{2-2} \cline{4-5} \cline{5-5} 
\end{tabular}
\par\end{centering}
\caption{\label{tb-dasc}Distribution of $\protect\dasc$ over $\mathfrak{D}_{n}$
for $0\protect\leq n\protect\leq9$.}
\end{table}
\begin{table}
\begin{centering}
\begin{tabular}{|c|c|c|c|c|}
\cline{1-2} \cline{2-2} \cline{4-5} \cline{5-5} 
$n$ & $D_{n}^{\ddes}(t)$ & \quad{} & $n$ & $D_{n}^{\ddes}(t)$\tabularnewline
\cline{1-2} \cline{2-2} \cline{4-5} \cline{5-5} 
$0$ & $1$ &  & $5$ & $31+9t+4t^{2}$\tabularnewline
\cline{1-2} \cline{2-2} \cline{4-5} \cline{5-5} 
$1$ & $0$ &  & $6$ & $160+66t+38t^{2}+t^{4}$\tabularnewline
\cline{1-2} \cline{2-2} \cline{4-5} \cline{5-5} 
$2$ & $1$ &  & $7$ & $910+622t+262t^{2}+54t^{3}+6t^{4}$\tabularnewline
\cline{1-2} \cline{2-2} \cline{4-5} \cline{5-5} 
$3$ & $2$ &  & $8$ & $6077+5254t+2781t^{2}+576t^{3}+144t^{4}+t^{6}$\tabularnewline
\cline{1-2} \cline{2-2} \cline{4-5} \cline{5-5} 
$4$ & $8+t^{2}$ &  & $9$ & $45026+49708t+27682t^{2}+9264t^{3}+1565t^{4}+243t^{5}+8t^{6}$\tabularnewline
\cline{1-2} \cline{2-2} \cline{4-5} \cline{5-5} 
\end{tabular}
\par\end{centering}
\caption{\label{tb-ddes}Distribution of $\protect\ddes$ over $\mathfrak{D}_{n}$
for $0\protect\leq n\protect\leq9$.}
\end{table}

\begin{proof}
For part (a), observe that each increasing run of length $k\geq2$
contributes $k-2$ double ascents. So, we take $G$ to be the graph
in Figure \ref{f-desar}, and we set 
\[
w_{1}^{(1,2)}=w_{1}^{(2,1)}=w_{1}^{(3,3)}=1\quad\text{and}\quad w_{k}^{(2,3)}=w_{k}^{(3,3)}=t^{k-2}\text{ for all }k\geq2.
\]
Then 
\begin{align*}
A & =\begin{bmatrix}1 & x & 0\\
x & 1 & {\displaystyle \frac{x^{2}}{1-tx}}\vphantom{{\displaystyle \frac{\frac{dy}{dx}}{\frac{dy}{dx}}}}\\
0 & 0 & {\displaystyle 1+x+\frac{x^{2}}{1-tx}}\vphantom{{\displaystyle \frac{\frac{dy}{dx}}{\frac{dy}{dx}}}}
\end{bmatrix}^{-1}=\begin{bmatrix}{\displaystyle \frac{1}{1-x^{2}}} & {\displaystyle \frac{-x}{1-x^{2}}} & {\displaystyle \frac{x^{3}}{(1-x)(1+x)(1+x(1+x)(1-t))}}\vphantom{{\displaystyle \frac{\frac{dy}{dx}}{\frac{dy}{dx}}}}\\
{\displaystyle \frac{-x}{1-x^{2}}} & {\displaystyle \frac{1}{1-x^{2}}} & {\displaystyle \frac{-x^{2}}{(1-x)(1+x)(1+x(1+x)(1-t))}}\vphantom{{\displaystyle \frac{\frac{dy}{dx}}{\frac{dy}{dx}}}}\\
0 & 0 & {\displaystyle \frac{1-tx}{1+x(1+x)(1-t)}}\vphantom{{\displaystyle \frac{\frac{dy}{dx}}{\frac{dy}{dx}}}}
\end{bmatrix},
\end{align*}
and converting to exponential generating functions gives 
\[
\begin{bmatrix}\cosh x & -\sinh x & {\displaystyle \frac{e^{(t-1)x/2}((1-t)\cosh(\frac{\alpha x}{2})+\alpha^{-1}(t^{2}-3)\sinh(\frac{\alpha x}{2}))}{3-2t}+\frac{e^{x}}{2(3-2t)}-\frac{e^{-x}}{2}}\vphantom{{\displaystyle \frac{\frac{dy}{dx}}{\frac{dy}{dx}}}}\\
-\sinh x & \cosh x & {\displaystyle \frac{e^{(t-1)x/2}\left((3+t)(2-t)\cosh(\frac{\alpha x}{2})+\alpha t\sinh(\frac{\alpha x}{2})\right)}{(3+t)(3-2t)}-\frac{e^{x}}{2(3-2t)}-\frac{e^{-x}}{2}}\vphantom{{\displaystyle \frac{\frac{dy}{dx}}{\frac{dy}{dx}}}}\\
0 & 0 & e^{(t-1)x/2}\left(\cosh(\frac{\alpha x}{2})-\alpha^{-1}(1+t)\sinh(\frac{\alpha x}{2})\right)
\end{bmatrix}
\]
for the matrix $\hat{A}$. Then we invert $\hat{A}$, take the $(1,3)$
entry of $(\hat{A})^{-1}$, and add $\cosh(x)$ to account for the
decreasing desarrangements (which have no double ascents) to obtain
our desired formula.

For (b), it suffices to count complements of desarrangements by double
ascents, so we take $G$ to be the graph in Figure \ref{f-cdesar}
and set 
\[
w_{k}^{(1,2)}=t^{k-2}\text{ for all even }k\geq2,\quad w_{1}^{(2,2)}=1,\quad\text{and }w_{k}^{(2,2)}=t^{k-2}\text{ for all }k\geq2.
\]
Thus, we have 
\[
A=\begin{bmatrix}1 & {\displaystyle \frac{x^{2}}{1-t^{2}x^{2}}}\vphantom{{\displaystyle \frac{\frac{dy}{dx}}{\frac{dy}{dx}}}}\\
0 & 1+x+{\displaystyle \frac{x^{2}}{1-tx}}\vphantom{{\displaystyle \frac{\frac{dy}{dx}}{\frac{dy}{dx}}}}
\end{bmatrix}^{-1}=\begin{bmatrix}1 & {\displaystyle \frac{-x^{2}}{(1+tx)(1+x(1+x)(1-t))}}\vphantom{{\displaystyle \frac{\frac{dy}{dx}}{\frac{dy}{dx}}}}\\
0 & {\displaystyle \frac{1-tx}{1+x(1+x)(1-t)}}\vphantom{{\displaystyle \frac{\frac{dy}{dx}}{\frac{dy}{dx}}}}
\end{bmatrix}
\]
and we obtain the desired formula by adding 1 to the $(1,2)$ entry
of $(\hat{A})^{-1}$; the computations are similar to the previous
ones and so we omit the remaining steps.
\end{proof}

\subsection{\label{ss-joint}Two joint distributions}

The run theorem can also be used to derive generating functions encoding
the joint distribution of multiple statistics, and we shall showcase
two examples. For our proofs in this section, we will only provide
the graph and the weights, omitting the details due to their similarity
to prior computations.
\begin{thm}
We have 
\begin{align*}
D^{(\pk,\des)}(s,t,x) & \coloneqq\sum_{n=0}^{\infty}\sum_{\pi\in\mathfrak{D}_{n}}s^{\pk(\pi)}t^{\des(\pi)}\frac{x^{n}}{n!}\\
 & =\frac{\mu(3-s-2t)\cosh\left(\frac{\mu x}{2}\right)-(1-s+(3-s)t-2t^{2})\sinh\left(\frac{\mu x}{2}\right)-\mu e^{(1-3t)x/2}}{(2-s-2t)\left(\mu\cosh\left(\frac{\mu x}{2}\right)-(1+t)\sinh\left(\frac{\mu x}{2}\right)\right)}
\end{align*}
where $\mu=\sqrt{1+2t(1-2s)+t^{2}}$.
\end{thm}

As expected, setting $s=1$ recovers Theorem \ref{t-des}, whereas
setting $t=1$ and then $s=t$ recovers Theorem \ref{t-pkval} (a).

\begin{proof}
We shall work with valleys and ascents in complements of desarrangements,
so take $G$ to be the directed graph in Figure \ref{f-cdesar}. Since
the number of valleys of a permutation is equal to its number of non-initial
long runs, we shall assign a weight of $s$ to every non-initial long
run, as well as a weight of $t^{k-1}$ to every run of length $k\geq1$.
As such, applying the generalized run theorem with the weights
\[
w_{k}^{(1,2)}=t^{k-1}\text{ for all even }k\geq2,\quad w_{1}^{(2,2)}=1,\quad\text{and }w_{k}^{(2,2)}=st^{k-1}\text{ for all }k\geq2
\]
yields the desired formula.
\end{proof}

Recall from Section \ref{ss-pkval} the notion of pixed points, and
let $\pix(\pi)$ denote the number of pixed points of $\pi$. One
sees directly that $\pix$ is equidistributed with $\fix$\textemdash the
number of fixed points\textemdash over $\mathfrak{S}_{n}$, and several
authors have studied joint distributions of $\pix$ along with other
statistics \cite{Burstein2015,Foata2008a,Foata2008,Lin2013}. The
next result appears to be new.
\begin{thm}
We have 
\begin{multline*}
P^{(\pix,\des)}(s,t,x)\coloneqq\sum_{n=0}^{\infty}\sum_{\pi\in\mathfrak{S}_{n}}s^{\pix(\pi)}t^{\des(\pi)}\frac{x^{n}}{n!}\\
=\frac{(1-t)\left((1-s)(1-s-t)te^{(s-t)x}+(1-t)\left((1-s)(s-t)e^{(s+t-1)x}-t(1-2t)\right)\right)}{(1-2t)(1-s-t)(s-t)\left(e^{(t-1)x}-t\right)}.
\end{multline*}
\end{thm}

Setting $s=1$ recovers the exponential generating function (\ref{e-Euleriangf})
for the Eulerian polynomials, setting $s=0$ recovers Theorem \ref{t-des},
and taking the limit as $t\rightarrow1$ recovers the well-known exponential
generating function $e^{(s-1)x}/(1-x)$ for the distribution of $\fix$
over $\mathfrak{S}_{n}$ \cite[A008290]{oeis}.
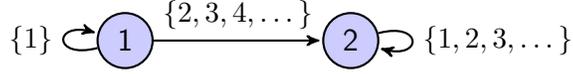
\begin{figure}
\begin{center}
\begin{tikzpicture}[->,>=stealth',shorten >=1pt,auto,node distance=3cm,   thick,main node/.style={circle,fill=blue!20,draw,font=\sffamily}]

\node[main node] (1) {1};
\node[main node] (2) [right of =1] {2};

\path[every node/.style={font=\sffamily\small}]
(1) edge [loop left] node {$\{1\}$} (1)   
(1) edge node {$\{2,3,4,\dots\}$} (2)
(2) edge [loop right] node {$\{1,2,3,\dots\}$} (2);

\end{tikzpicture}
\end{center}

\caption{\label{f-pix}Directed graph for counting permutations by pixed points.}
\end{figure}

\begin{proof}
Take $G$ to be the directed graph in Figure \ref{f-pix}. Note that
the descent composition of any permutation is either $(1,1)$-admissible
or $(1,2)$-admissible, depending on whether the permutation has a
long run. If we let $\pi=\iota\delta$ be the pixed factorization
of $\pi$, then $\pi^{c}=\tilde{\iota}\tilde{\delta}$ where $\tilde{\iota}$
is a decreasing sequence and $\tilde{\delta}$ is the complement of
a desarrangement, so we wish to weight each letter in $\tilde{\iota}$
by $s$ and each ascent of $\pi^{c}$ by $t$. If $\pi_{1}^{c}\pi_{2}^{c}\cdots\pi_{j}^{c}$
is the maximal decreasing prefix of $\pi^{c}$, then $\tilde{\iota}=\pi_{1}^{c}\pi_{2}^{c}\cdots\pi_{j}^{c}$
if $\pi_{j+1}^{c}\pi_{j+2}^{c}\cdots\pi_{n}^{c}$ is the complement
of a desarrangement, and $\tilde{\iota}=\pi_{1}^{c}\pi_{2}^{c}\cdots\pi_{j-1}^{c}$
otherwise; the latter holds precisely when the first long run of $\pi^{c}$
is odd. Thus, we apply the generalized run theorem with the weights
\[
w_{1}^{(1,1)}=s,\quad w_{k}^{(1,2)}=\begin{cases}
t^{k-1}, & \text{for all even }k\geq2,\\
st^{k-1}, & \text{for all odd }k\geq3,
\end{cases}\quad\text{and }w_{k}^{(2,2)}=t^{k-1}\text{ for all }k\geq1,
\]
adding up the $(1,1)$ and $(1,2)$ entries of $(\hat{A})^{-1}$.
\end{proof}

The same approach can be used to determine the joint distribution
of $\pix$ with the other statistics considered earlier, but we will
not do this here.

\section{\label{s-pa}Pattern avoidance in desarrangements}

We now turn our attention to the study of pattern avoidance in desarrangements.
Given a set $\Pi$ of patterns, recall that $\mathfrak{D}_{n}(\Pi)$
denotes the set of permutations in $\mathfrak{D}_{n}$ avoiding all
patterns in $\Pi$, and that $d_{n}(\Pi)=\left|\mathfrak{D}_{n}(\Pi)\right|$.
We refer to the $\mathfrak{D}_{n}(\Pi)$ as \textit{desarrangement
avoidance classes}. In this section, we will determine $d_{n}(\Pi)$
for all pattern sets $\Pi\subseteq\mathfrak{S}_{3}$ using a mixture
of combinatorial and generating function techniques. See Table \ref{tb-pa}
for a summary of our results.
\begin{table}
\begin{centering}
\renewcommand{\arraystretch}{1.4}%
\begin{tabular}{|>{\raggedright}p{1.1in}|>{\raggedright}m{2.25in}|>{\raggedright}p{1in}|>{\raggedright}p{0.7in}|}
\hline 
$\Pi$ & $d_{n}(\Pi)$ & Reference & OEIS \cite{oeis}\tabularnewline
\hline 
$\{321\}$ & $C_{n-1}$ (Catalan numbers) & Theorem \ref{t-321} & A000108\tabularnewline
\hline 
$\{132\}$ & \multirow{3}{2.25in}{$F_{n+1}$ (Fine numbers)} & \multirow{2}{1in}{Theorem \ref{t-Fine1}} & \multirow{3}{0.7in}{A000957}\tabularnewline
\cline{1-1} 
$\{231\}$ &  &  & \tabularnewline
\cline{1-1} \cline{3-3} 
$\{123\}$ &  & Theorem \ref{t-Fine2} & \tabularnewline
\hline 
$\{213\}$ & \multirow{2}{2.25in}{${\displaystyle \sum_{k=1}^{n-1}(-1)^{n-k-1}C_{k}}$} & \multirow{2}{1in}{Theorem \ref{t-213-312}} & \multirow{2}{0.7in}{A033297}\tabularnewline
\cline{1-1} 
$\{312\}$ &  &  & \tabularnewline
\hline 
$\{132,321\}$ & \multirow{2}{2.25in}{$n-1$} & \multirow{2}{1in}{Theorem \ref{t-132+321}} & \multirow{2}{0.7in}{A000027}\tabularnewline
\cline{1-1} 
$\{132,231,321\}$ &  &  & \tabularnewline
\hline 
$\{132,231\}$ & \multirow{2}{2.25in}{$2^{n-2}$} & Theorem \ref{t-132+231} & \multirow{3}{0.7in}{A000079}\tabularnewline
\cline{1-1} \cline{3-3} 
$\{231,321\}$ &  & Theorem \ref{t-321+231} & \tabularnewline
\cline{1-3} \cline{2-3} \cline{3-3} 
$\{312,321\}$ & $2^{n-3}$ & Theorem \ref{t-312+321} & \tabularnewline
\hline 
$\{123,213\}$ & \multirow{5}{2.25in}{$J_{n-1}$ (Jacobsthal numbers)} & \multirow{3}{1in}{Theorem \ref{t-213+friends}} & \multirow{5}{0.7in}{A001045}\tabularnewline
\cline{1-1} 
$\{132,213\}$ &  &  & \tabularnewline
\cline{1-1} 
$\{213,231\}$ &  &  & \tabularnewline
\cline{1-1} \cline{3-3} 
$\{132,312\}$ &  & \multirow{2}{1in}{Theorem \ref{t-312+friends}} & \tabularnewline
\cline{1-1} 
$\{231,312\}$ &  &  & \tabularnewline
\hline 
$\{123,132\}$$\vphantom{{\displaystyle \frac{\frac{b}{b}}{\frac{b}{b}}}}$ & ${\displaystyle \frac{2^{n+1}+(7-3n)(-1)^{n}}{9}}$ & Theorem \ref{t-123+132} & A113954\tabularnewline
\hline 
$\{123,231\}$$\vphantom{{\displaystyle \frac{\frac{b}{b}}{\frac{b}{b}}}}$ & ${\displaystyle \frac{2n(n-1)+(5-2n)(-1)^{n}+3}{8}}$ & Theorem \ref{t-123+231} & A130404\tabularnewline
\hline 
$\{123,312\}$$\vphantom{{\displaystyle \frac{\frac{b}{b}}{\frac{b}{\frac{a}{b}}}}}$ & ${\displaystyle \left\lceil \frac{(n-1)^{2}}{4}\right\rceil }$ & Theorem \ref{t-123+312} & A004652\tabularnewline
\hline 
$\{123,132,231\}$$\vphantom{{\displaystyle \frac{\frac{\frac{b}{}}{b}}{\frac{b}{\frac{\frac{b}{}}{b}}}}}$ & $\begin{cases}
n-1, & \text{if }n\text{ is odd,}\\
1, & \text{if }n\text{ is even.}
\end{cases}$ & Theorem \ref{t-123+132+231} & A124625\tabularnewline
\hline 
$\{123,132,312\}$ & \multirow{5}{2.25in}{${\displaystyle \left\lfloor \frac{n}{2}\right\rfloor }$} & \multirow{2}{1in}{Theorem \ref{t-123+312+friends}} & \multirow{5}{0.7in}{A008619}\tabularnewline
\cline{1-1} 
$\{123,231,312\}$ &  &  & \tabularnewline
\cline{1-1} \cline{3-3} 
$\{123,213,231\}$ &  & \multirow{2}{1in}{Theorem \ref{t-213+231+friends}} & \tabularnewline
\cline{1-1} 
$\{132,213,231\}$ &  &  & \tabularnewline
\cline{1-1} \cline{3-3} 
$\{132,231,312\}$ &  & Theorem \ref{t-132+231+312} & \tabularnewline
\hline 
$\{123,132,213\}$ & \multirow{2}{2.25in}{$f_{n-1}$ (Fibonacci numbers)} & Theorem \ref{t-123+132+213} & \multirow{2}{0.7in}{A000045}\tabularnewline
\cline{1-1} \cline{3-3} 
$\{231,312,321\}$ &  & Theorem \ref{t-231+312+321} & \tabularnewline
\hline 
\end{tabular}
\par\end{centering}
\caption{\label{tb-pa}Summary of pattern avoidance results.}
\end{table}

Some of our proofs will use generating functions for the numbers $d_{n}(\Pi)$
and $s_{n}(\Pi)\coloneqq\left|\mathfrak{S}_{n}(\Pi)\right|$. To that
end, we define
\[
D(x;\Pi)\coloneqq\sum_{n=0}^{\infty}d_{n}(\Pi)x^{n}\quad\text{and}\quad S(x;\Pi)\coloneqq\sum_{n=0}^{\infty}s_{n}(\Pi)x^{n}.
\]

One remark is in order before we continue. In the permutation patterns
literature, two pattern sets $\Pi$ and $\Pi^{\prime}$ are said to
be \textit{Wilf equivalent} if $s_{n}(\Pi)=s_{n}(\Pi^{\prime})$ for
all $n\geq0$. Note that $\pi$ avoids a pattern $\sigma$ if and
only if its complement $\pi^{c}$ avoids the complement $\sigma^{c}$
of $\sigma$. Therefore, if $\Pi^{c}$ is the set of patterns obtained
by taking the complement of every pattern in $\Pi$, then $\Pi$ and
$\Pi^{c}$ are Wilf equivalent for all $\Pi$. The same is true for
other symmetries such as reversal, reverse-complementation, and the
inverse, and Wilf equivalences achieved via such symmetries are called
trivial. However, in the setting of desarrangements, all Wilf equivalences
are nontrivial because the property of being a desarrangement is not
invariant under any of these symmetries.

\subsection{Preliminary lemmas and trivial results}

Below is a preliminary lemma about pattern avoidance in desarrangements
which will prove to be helpful for what lies ahead.
\begin{lem}
\label{l-desav}Let $n\geq2$ and $\pi\in\mathfrak{D}_{n}$.
\begin{enumerate}
\item [\normalfont{(a)}] If $\pi$ avoids 213, then $\pi_{1}=n$.
\item [\normalfont{(b)}] If $\pi$ avoids 231, then the letter 1 is at
the first ascent of $\pi$.
\item [\normalfont{(c)}] If $\pi$ avoids 312, then $\pi_{1}=\pi_{2}+1$.
\item [\normalfont{(d)}] If $\pi$ avoids 321, then $\pi_{2}=1$.
\end{enumerate}
\end{lem}

\begin{proof}
Any desarrangement $\pi$ has $\pi_{1}>\pi_{2}$ (or else $1$ would
be the first ascent of $\pi$), so $\pi_{2}\neq n$. If the letter
$n$ appears after $\pi_{2}$ in $\pi$, then $\pi_{1}\pi_{2}n$ would
be an occurrence of 213. Thus, any $\pi\in\mathfrak{D}_{n}(213)$
has $\pi_{1}=n$, and (a) is proven.

If $\pi$ avoids 231, then there cannot be an ascent before the letter
1 in $\pi$, and (b) follows.

If $\pi_{1}>\pi_{2}+1$, then the letter $\pi_{2}+1$ would appear
after $\pi_{2}$ in $\pi$, so $\pi_{1}\pi_{2}(\pi_{2}+1)$ would
be an occurrence of 312. Also, as observed above, no desarrangement
$\pi$ can have $\pi_{1}<\pi_{2}$. Thus, any $\pi\in\mathfrak{D}_{n}(312)$
has $\pi_{1}=\pi_{2}+1$, proving (c).

Finally, if 1 appears after $\pi_{2}$ in $\pi$, then $\pi_{1}\pi_{2}1$
would be an occurrence of 321 in $\pi$. Since $\pi_{1}=1$ would
mean that the first ascent of $\pi$ is 1, any $\pi\in\mathfrak{D}_{n}(321)$
must have $\pi_{2}=1$, which proves (d).
\end{proof}

We now state several results about desarrangement avoidance classes
that have at most one permutation per length $n$. Together with the
more interesting results presented in Table~\ref{tb-pa}, these complete
the enumeration of $\mathfrak{D}_{n}(\Pi)$ for all length 3 pattern
sets $\Pi\subseteq\mathfrak{S}_{3}$.

\begin{prop}
\label{p-trivf}If $\{123,321\}\subseteq\Pi$, then $d_{n}(\Pi)=0$
for all $n\geq5$.
\end{prop}

\begin{proof}
This is an immediate consequence of the Erd\H{o}s\textendash Szekeres
theorem \cite{Erdoes1935}.
\end{proof}

\begin{prop}
If $\{213,312\}\subseteq\Pi\subseteq\mathfrak{S}_{3}$, then $d_{n}(\Pi)=0$
for all odd $n$ and 
\[
d_{n}(\Pi)=\begin{cases}
1, & \text{if }321\notin\Pi,\\
0, & \text{if }321\in\Pi,
\end{cases}
\]
 for all even $n\geq4$.
\end{prop}

\begin{proof}
Any permutation $\pi\in\mathfrak{S}_{n}(213,312)$ is of the form
$\pi=\tau n\rho$ where $\tau$ is an increasing sequence and $\rho$
is a decreasing sequence. In particular, if $\pi\in\mathfrak{D}_{n}(213,312)$,
then $\tau$ must be empty by Lemma \ref{l-desav}\ (a), so $\pi$
is the decreasing permutation $\pi=n\cdots21$, which in turn implies
that $n$ is even.
\end{proof}

\begin{prop}
If $\{213,321\}\subseteq\Pi\subseteq\mathfrak{S}_{3}$, then 
\[
d_{n}(\Pi)=\begin{cases}
1, & \text{if }123\notin\Pi\text{ and }312\notin\Pi,\\
0, & \text{otherwise.}
\end{cases}
\]
 for all $n\geq3$.
\end{prop}

\begin{proof}
By Lemma \ref{l-desav}\ (a) and (d), any permutation in $\pi\in\mathfrak{D}_{n}(213,321)$
is of the form $\pi=n1\tau$, which then implies that $\tau=23\cdots(n-1)$
or else $\pi$ would contain an occurrence of $321$. In other words,
the only permutation in $\mathfrak{D}_{n}(213,321)$ is $n12\cdots(n-1)$.
\end{proof}

\begin{prop}
If $\{132,312,321\}\subseteq\Pi\subseteq\mathfrak{S}_{3}$, then 
\[
d_{n}(\Pi)=\begin{cases}
1, & \text{if }123\notin\Pi\text{ and }213\notin\Pi,\\
0, & \text{otherwise.}
\end{cases}
\]
 for all $n\geq3$.
\end{prop}

\begin{proof}
By Lemma \ref{l-desav}\ (c) and (d), any permutation in $\pi\in\mathfrak{D}_{n}(312,321)$
is of the form $\pi=21\tau$, and $\pi$ avoiding 132 forces $\tau=34\cdots n$.
In other words, the only permutation in $\mathfrak{D}_{n}(132,312,321)$
is $2134\cdots n$.
\end{proof}

\begin{prop}
\label{p-trivl}If $\Pi=\{123,132,231,\sigma\}$ where $\sigma\in\{213,312\}$,
then $d_{n}(\Pi)=1$ for all $n\geq3$.
\end{prop}

\begin{proof}
Let $\pi\in\mathfrak{D}_{n}(123,132,231)$. By Lemma \ref{l-desav}\ (b),
we know that the letter 1 is at the first ascent of $\pi$. Additionally,
since $\pi$ avoids 123 and 132, there can be at most one letter in
$\pi$ appearing after 1, so either $\pi_{n-1}=1$ or $\pi_{n}=1$.
Since the first ascent of $\pi$ is even, this means that $\pi_{n-1}=1$
if $n$ is odd and $\pi_{n}=1$ if $n$ is even. When $n$ is even,
$\pi$ is the decreasing permutation $\pi=n\cdots21$. When $n$ is
odd, $\pi$ is of the form $\pi=\tau1\pi_{n}$, where $\tau$ consists
of the $n-2$ remaining letters arranged in decreasing order. Avoiding
$\sigma=213$ forces $\pi=n(n-1)\cdots312$, whereas avoiding $\sigma=312$
forces $\pi=(n-1)(n-2)\cdots21n$.
\end{proof}

\subsection{Single restrictions}

Let us count desarrangements avoiding a single pattern of length 3.
It is well-known that $s_{n}(\sigma)=C_{n}$\textemdash the $n$th
\textit{Catalan number} \cite[A000108]{oeis}\textemdash for every
$\sigma\in\mathfrak{S}_{3}$ \cite[Section 4.2]{Bona2022}. The situation,
however, is more subtle upon restricting to desarrangements. Nonetheless,
there is one desarrangement avoidance class whose enumeration is still
given by the Catalan numbers.
\begin{thm}
\label{t-321}For all $n\geq2$, we have $d_{n}(321)=C_{n-1}$.
\end{thm}

\begin{proof}
We present a bijection from $\mathfrak{S}_{n-1}(321)$ to $\mathfrak{D}_{n}(321)$.
Given $\pi\in\mathfrak{S}_{n-1}(321)$, let $\pi^{\prime}$ be the
permutation obtained by adding 1 to each letter of $\pi$ and then
inserting 1 immediately after $\pi_{1}$. For example, if $\pi=45123$,
then $\pi^{\prime}=516234$. It is clear that $\pi\mapsto\pi^{\prime}$
cannot introduce an occurrence of 321 and that the first ascent of
$\pi^{\prime}$ is 2, so $\pi^{\prime}\in\mathfrak{D}_{n}(321)$.
To undo this procedure, we take $\pi\in\mathfrak{D}_{n}(321)$, remove
the letter $1$, and subtract $1$ from each remaining letter of $\pi$;
Lemma \ref{l-desav}\ (d) guarantees that this is a well-defined
inverse of $\pi^{\prime}\mapsto\pi$. Since the permutations in $\mathfrak{S}_{n-1}(321)$
are counted by $C_{n-1}$, the result follows.
\end{proof}

Next, we show that $\mathfrak{D}_{n}(123)$, $\mathfrak{D}_{n}(132)$,
and $\mathfrak{D}_{n}(231)$ are all counted by the \textit{Fine numbers}
$F_{n}$, which can be defined through the generating function
\begin{equation}
\sum_{n=0}^{\infty}F_{n}x^{n}=\frac{1-\sqrt{1-4x}}{3-\sqrt{1-4x}}\label{e-Finegf}
\end{equation}
or equivalently through their connection with the Catalan numbers:
$C_{n}=2F_{n+1}+F_{n}$ for all $n\geq1$. The first several Fine
numbers are displayed in the following table:
\begin{center}
\begin{tabular}{c|c|c|c|c|c|c|c|c|c|c|c|c}
$n$ & $0$ & $1$ & $2$ & $3$ & $4$ & $5$ & $6$ & $7$ & $8$ & $9$ & $10$ & $11$\tabularnewline
\hline 
$F_{n}$ & $0$ & $1$ & $0$ & $1$ & $2$ & $6$ & $18$ & $57$ & $186$ & $622$ & $2120$ & $7338$\tabularnewline
\end{tabular}
\par\end{center}

\noindent See entry A000957 of the OEIS \cite{oeis}, as well as the
survey article \cite{Deutsch2001} for a sample of the many combinatorial
interpretations of $F_{n}$. 

For our proof of the next result, it will be convenient for us to
consider the shifted Fine generating function 
\begin{equation}
\sum_{n=0}^{\infty}F_{n+1}x^{n}=\frac{2}{1+2x+\sqrt{1-4x}}\label{e-Finegfsh}
\end{equation}
obtained by dividing (\ref{e-Finegf}) by $x$. We shall also make
use of the Catalan generating function 
\begin{equation}
C\coloneqq\sum_{n=0}^{\infty}C_{n}x^{n}=\frac{1-\sqrt{1-4x}}{2x}.\label{e-Catgf}
\end{equation}

\begin{thm}
\label{t-Fine1}For all $n\geq2$ and $\sigma\in\{132,231\}$, we
have $d_{n}(\sigma)=F_{n+1}$.
\end{thm}

\begin{proof}
We first prove the result for $\sigma=231$. For convenience, let
$D$ denote the generating function $D(x;231)$. We claim that $D$
satisfies the functional equation
\begin{equation}
D=1+x\left(D-1\right)C+x\left(C-D\right).\label{e-Finefe}
\end{equation}
The empty permutation contributes 1 to the right-hand side of (\ref{e-Finefe}).
Every nonempty $231$-avoiding desarrangement $\pi$ can be written
as the concatenation 
\[
\pi=\tau n\rho,
\]
where $\tau\in\mathfrak{D}_{k}(231)$ for some $0\leq k\leq n-1$,
and $\rho$ is a 231-avoiding permutation on the letters $\{k+1,k+2,\dots,n-1\}$.
We further divide into two cases. If $\tau$ is nonempty, then this
case contributes the term $x\left(D-1\right)C$ to (\ref{e-Finefe}),
since $C$ is the generating function for 231-avoiding permutations.
On the other hand, if $\tau$ is empty, then $\rho$ is a non-desarrangement;
this case contributes $x\left(C-D\right)$ to (\ref{e-Finefe}). Hence,
(\ref{e-Finefe}) is proven. Simplifying (\ref{e-Finefe}) yields
\[
D=1+xD\left(C-1\right);
\]
then, substituting in (\ref{e-Catgf}) and solving for $D$ gives
us
\[
D=\frac{2}{1+2x+\sqrt{1-4x}}
\]
which, upon comparing with (\ref{e-Finegfsh}), completes the proof.

The proof for $\sigma=132$ is essentially the same; the only difference
is that the letters of $\rho$ are now smaller than the letters of
$\tau$, but this leads to the same functional equation and generating
function as above.
\end{proof}

\begin{thm}
\label{t-Fine2}For all $n\geq2$, we have $d_{n}(123)=F_{n+1}$.
\end{thm}

\begin{proof}
It is readily verified that Simion and Schmidt's famous bijection
between 123- and 132-avoiding permutations \cite[Proposition 19]{Simion1985}
restricts to a bijection between $\mathfrak{D}_{n}(123)$ and $\mathfrak{D}_{n}(132)$;
we omit the details. Thus $d_{n}(123)=d_{n}(132)=F_{n+1}$.
\end{proof}

The last two singleton desarrangement avoidance classes, $\mathfrak{D}_{n}(213)$
and $\mathfrak{D}_{n}(312)$, are counted by a sequence defined recursively
as follows: let $a_{0}=1$, and $a_{n+1}=C_{n}-a_{n}$ for all $n\geq0$.
It can be shown inductively that 
\[
a_{n}=\sum_{k=1}^{n-1}(-1)^{n-k-1}C_{k}
\]
for all $n\geq1$. The following table displays the first several
values of the sequence $\{a_{n}\}$:
\begin{center}
\begin{tabular}{c|c|c|c|c|c|c|c|c|c|c|c|c}
$n$ & $0$ & $1$ & $2$ & $3$ & $4$ & $5$ & $6$ & $7$ & $8$ & $9$ & $10$ & $11$\tabularnewline
\hline 
$a_{n}$ & $1$ & $0$ & $1$ & $1$ & $4$ & $10$ & $32$ & $100$ & $329$ & $1101$ & $3761$ & $3761$\tabularnewline
\end{tabular}
\par\end{center}

\noindent See also entry A033297 of the OEIS \cite{oeis}.
\begin{thm}
\label{t-213-312}For all $n\geq0$ and $\sigma\in\{213,312\}$, we
have $d_{n}(\sigma)=a_{n}$.
\end{thm}

This theorem can be readily proven using generating functions \`{a}
la Theorem \ref{t-Fine1}, but we will instead give a combinatorial
proof.
\begin{lem}
\label{l-213}We have $C_{n}=d_{n}(213)+d_{n+1}(213)$ for all $n\geq0$.
\end{lem}

\begin{proof}
It suffices to give a bijection from $\mathfrak{S}_{n}(213)$ to $\mathfrak{D}_{n}(213)\sqcup\mathfrak{D}_{n+1}(213)$.
Given $\pi\in\mathfrak{S}_{n}(213)$, let
\[
\varphi_{213}(\pi)\coloneqq\begin{cases}
\pi, & \text{if }\pi\text{ is a desarrangement},\\
(n+1)\pi, & \text{otherwise}.
\end{cases}
\]
In other words, $\varphi_{213}$ is the identity map on $\mathfrak{D}_{n}(213)$,
and sends each non-desarrangement in $\mathfrak{S}_{n}(213)$ to the
desarrangement in $\mathfrak{D}_{n+1}(213)$ obtained by prepending
the letter $n+1$. We know from Lemma \ref{l-desav}\ (a) that every
permutation in $\mathfrak{D}_{n+1}(213)$ begins with its largest
letter $n+1$. Thus, the inverse of $\varphi_{213}$ is given by removing
the first letter $n+1$ if we begin with a permutation in $\mathfrak{D}_{n+1}(213)$,
and to do nothing if we begin with a permutation in $\mathfrak{D}_{n}(213)$.
\end{proof}

\begin{lem}
\label{l-312}We have $C_{n}=d_{n}(312)+d_{n+1}(312)$ for all $n\geq0$.
\end{lem}

\begin{proof}
We give a bijection from $\mathfrak{S}_{n}(312)$ to $\mathfrak{D}_{n}(312)\sqcup\mathfrak{D}_{n+1}(312)$.
For convenience, let $\pi^{\prime}$ denote the word obtained from
a permutation $\pi$ by adding 1 to every letter greater than $\pi_{1}$.
For example, if $\pi=342561$ then $\pi^{\prime}=352671$. Then define
$\varphi_{312}(\pi)\colon\mathfrak{S}_{n}(312)\rightarrow\mathfrak{D}_{n}(312)\sqcup\mathfrak{D}_{n+1}(312)$
by
\[
\varphi_{312}(\pi)\coloneqq\begin{cases}
\pi, & \text{if }\pi\text{ is a desarrangement},\\
(\pi_{1}+1)\pi^{\prime}, & \text{otherwise.}
\end{cases}
\]
Continuing the above example, we have $\varphi_{312}(342561)=4352671$.
If applying $\varphi_{312}$ to $\pi\in\mathfrak{S}_{n}(312)$ introduces
an occurrence of 312, then that occurrence must be of the form $(\pi_{1}+1)\pi_{j}\pi_{k}$\textemdash that
is, it must have $\pi_{1}+1$ as its first letter. But then $\pi_{1}\pi_{j}\pi_{k}$
would be an occurrence of 312 in $\pi$, a contradiction. Hence, $\varphi_{312}$
is a well-defined map from $\mathfrak{S}_{n}(312)$ to $\mathfrak{D}_{n}(312)\sqcup\mathfrak{D}_{n+1}(312)$.

Now, recall from Lemma \ref{l-desav}\ (c) that every 312-avoiding
desarrangement $\pi$ satisfies $\pi_{1}=\pi_{2}+1$. Therefore, the
inverse of $\varphi_{312}$ is given by 
\[
\varphi_{312}^{-1}(\pi)=\begin{cases}
\pi, & \text{if }\left|\pi\right|=n,\\
\tilde{\pi}, & \text{if }\left|\pi\right|=n+1,
\end{cases}
\]
where $\tilde{\pi}$ is the permutation obtained from $\pi$ by subtracting
1 from every letter greater than $\pi_{2}$ and then removing the
first letter $\pi_{1}$.
\end{proof}

Lemmas \ref{l-213}\textendash \ref{l-312} show that $d_{n}(213)$
and $d_{n}(312)$ satisfy the same recurrence as $a_{n}$, and Theorem~\ref{t-213-312} follows upon making the observation that all three
sequences share the same initial term.

\subsection{Double restrictions}

We now move on to the desarrangement avoidance classes $\mathfrak{D}_{n}(\Pi)$
where $\Pi$ consists of two patterns of length 3.
\begin{thm}
\label{t-132+321}For all $n\geq1$, we have $d_{n}(132,321)=n-1$.
\end{thm}

\begin{proof}
The result for $n=1$ is trivial, so let $\pi\in\mathfrak{D}_{n}(132,321)$
where $n\geq2$. By Lemma~\ref{l-desav}~(d), we know that $\pi_{2}=1$.
Moreover, $\pi_{3}\pi_{4}\cdots\pi_{n}$ is an increasing sequence
since $\pi$ avoids 132. We now see that $\pi$ is completely determined
by the choice of $\pi_{1}$, as all other letters must be placed after
$\pi_{1}$ in increasing order. Since there are $n-1$ choices for
$\pi_{1}$\textemdash namely $2,3,\dots,n$\textemdash the result
follows.
\end{proof}

From the description given in the above proof, we see that all permutations
in $\mathfrak{D}_{n}(132,321)$ also avoid 231, so in fact $\mathfrak{D}_{n}(132,321)=\mathfrak{D}_{n}(132,231,321)$
and thus 
\[
d_{n}(132,231,321)=d_{n}(132,321)=n-1
\]
for all $n\geq1$.

Simion and Schmidt \cite[Section 3]{Simion1985} showed that $s_{n}(\Pi)=2^{n-1}$
for several sets $\Pi\subseteq\mathfrak{S}_{3}$ with $\left|\Pi\right|=2$.
We show that two of these sets satisfy $d_{n}(\Pi)=2^{n-2}$, and
another one satisfies $d_{n}(\Pi)=2^{n-3}$. The first two of these
results are obtained by constructing an involution on $\mathfrak{S}_{n}(\Pi)$
which restricts to a bijection between the desarrangements and the
non-desarrangements in that class, thus showing that exactly half
of the permutations in $\mathfrak{S}_{n}(\Pi)$ are desarrangements.
Our proof of the third result uses a direct bijection between $\mathfrak{D}_{n}(\Pi)$
and $\mathfrak{S}_{n-2}(\Pi)$.
\begin{thm}
\label{t-132+231}For all $n\geq2$, we have $d_{n}(132,231)=2^{n-2}$.
\end{thm}

\begin{proof}
Observe that the permutations in $\mathfrak{S}_{n}(132,231)$ are
precisely those of the form $\tau1\rho$ where $\tau$ is a decreasing
sequence and $\rho$ is an increasing sequence. Since there are two
choices for each of the $n-1$ letters in $\{2,3,\dots,n\}$\textemdash either
it belongs to $\tau$ or belongs to $\rho$\textemdash we recover
Simion and Schmidt's result that $s_{n}(132,231)=2^{n-1}$ \cite[Proposition 9]{Simion1985}.

Consider the involution $\varphi_{132,231}$ on $\mathfrak{S}_{n}(132,231)$
defined by 
\[
\varphi_{132,231}(\pi)\coloneqq\begin{cases}
\pi_{2}\pi_{3}\cdots\pi_{n}\pi_{1}, & \text{if }\pi_{1}=n,\\
\pi_{n}\pi_{1}\pi_{2}\cdots\pi_{n-1}, & \text{if }\pi_{n}=n.
\end{cases}
\]
(In other words, we are toggling the inclusion of $n$ in $\tau$
versus $\rho$.) Then $\varphi_{132,231}$ sends desarrangements to
non-desarrangements and vice versa, thus showing that exactly half
of the $2^{n-1}$ permutations in $\mathfrak{S}_{n}(132,231)$ are
desarrangements.
\end{proof}

\begin{thm}
\label{t-321+231}For all $n\geq2$, we have $d_{n}(231,321)=2^{n-2}$.
\end{thm}

\begin{proof}
We know that $s_{n}(231,321)=2^{n-1}$ \cite[Lemma 5\ (a) and Proposition 7]{Simion1985}.
Observe that the permutations $\pi\in\mathfrak{S}_{n}(231,321)$ either
have $\pi_{1}=1$ or $\pi_{2}=1$, and we know from Lemma~\ref{l-desav}~(d)
that the desarrangements are precisely the ones with $\pi_{2}=1$.
Hence, the involution on $\mathfrak{S}_{n}(231,321)$ defined by swapping
$\pi_{1}$ and $\pi_{2}$ is a bijection between desarrangements and
non-desarrangements in $\mathfrak{S}_{n}(231,321)$, so exactly half
of the $2^{n-1}$ permutations in $\mathfrak{S}_{n}(231,321)$ are
desarrangements.
\end{proof}

\begin{thm}
\label{t-312+321}For all $n\geq3$, we have $d_{n}(312,321)=2^{n-3}$.
\end{thm}

\begin{proof}
Since $s_{n}(312,321)=2^{n-1}$ for all $n\geq1$ \cite[Lemma 5 (a)\ and Proposition 7]{Simion1985},
it suffices to give a bijection between $\mathfrak{D}_{n}(312,321)$
and $\mathfrak{S}_{n-2}(312,321)$. From Lemma \ref{l-desav}\ (c)
and (d), we know that any $\pi\in\mathfrak{D}_{n}(312,321)$ is of
the form $\pi=21\tau$ where $\tau$ is a 312- and 321-avoiding permutation
on the letters $\{3,4,\dots,n\}$. So our desired bijection is given
by $\pi\mapsto\std(\tau)$.
\end{proof}

The next five desarrangement avoidance classes are enumerated by the
\textit{Jacobsthal numbers} $J_{n}$, which are defined by the recurrence
relation $J_{n}=J_{n-1}+2J_{n-2}$ for all $n\geq2$ along with initial
conditions $J_{0}=0$ and $J_{1}=1$. Their generating function is
given by
\[
\sum_{n=0}^{\infty}J_{n}x^{n}=\frac{x}{(1+x)(1-2x)},
\]
although it will be more convenient for us to work with the shifted
generating function 
\begin{equation}
1+\sum_{n=1}^{\infty}J_{n-1}x^{n}=\frac{1-x-x^{2}}{(1+x)(1-2x)}.\label{e-Jacobgf}
\end{equation}
The first several Jacobsthal numbers are displayed below:
\begin{center}
\begin{tabular}{c|c|c|c|c|c|c|c|c|c|c|c|c}
$n$ & $0$ & $1$ & $2$ & $3$ & $4$ & $5$ & $6$ & $7$ & $8$ & $9$ & $10$ & $11$\tabularnewline
\hline 
$J_{n}$ & $0$ & $1$ & $1$ & $3$ & $5$ & $11$ & $21$ & $43$ & $85$ & $171$ & $341$ & $683$\tabularnewline
\end{tabular}
\par\end{center}

\noindent See also OEIS entry A001045 \cite{oeis}.
\begin{thm}
\label{t-213+friends}For all $n\geq1$ and $\sigma\in\{123,132,231\}$,
we have $d_{n}(213,\sigma)=J_{n-1}$.
\end{thm}

\begin{proof}
Given any $\sigma\in\{123,132,231\}$, we have $s_{n}(213,\sigma)=2^{n-1}$
for all $n\geq1$ \cite[Section 3]{Simion1985}, so
\[
S(x;213,\sigma)=1+\sum_{n=1}^{\infty}2^{n-1}x^{n}=\frac{1-x}{1-2x}.
\]
Now, recall from Lemma \ref{l-desav}\ (a) that any 213-avoiding
desarrangement begins with its largest letter. Then any nonempty $\pi\in\mathfrak{D}_{n}(213,\sigma)$
is of the form $\pi=n\tau$ where $\tau$ is a non-desarrangement
in $\mathfrak{S}_{n-1}(213,\sigma)$. Letting $D$ denote the generating
function $D(x;213,\sigma)$, we thus have the functional equation
\[
D=1+x\left(\frac{1-x}{1-2x}-D\right)
\]
which, upon solving, gives us 
\[
D=\frac{1-x-x^{2}}{(1+x)(1-2x)}.
\]
By (\ref{e-Jacobgf}), this is precisely the generating function for
$J_{n-1}$, and we are done.
\end{proof}

\begin{thm}
\label{t-312+friends}For all $n\geq1$ and $\sigma\in\{132,231\}$,
we have $d_{n}(312,\sigma)=J_{n-1}$.
\end{thm}

\begin{proof}
Take $\sigma=231$, and let $D$ denote the generating function $D(x;231,312)$.
We claim that $D$ satisfies the functional equation
\begin{equation}
D=\frac{x(D-1)}{1-x}+\frac{1}{1-x^{2}}.\label{e-312fe}
\end{equation}
Observe that any nonempty $\pi\in\mathfrak{D}_{n}(231,312)$ is of
the form $\pi=\tau n\rho$ where $\tau\in\mathfrak{D}_{k}(231,312)$
for some $0\leq k\leq n-1$ and $\rho$ consists of the letters $k+1,k+2,\dots,n-1$
arranged in decreasing order. The case where $k\geq1$ gives the contribution
$x(D-1)/(1-x)$ to (\ref{e-312fe}). If $k=0$, then $\pi$ itself
must be the decreasing permutation $n\cdots21$, but in this case
$\pi$ is a desarrangement only if $n$ is even. Together with the
case where $\pi$ is empty, this gives the contribution $1/(1-x^{2})$
to (\ref{e-312fe}). Solving (\ref{e-312fe}) for $D$ yields 
\[
D=\frac{1-x-x^{2}}{(1+x)(1-2x)}
\]
which, upon comparing with (\ref{e-Jacobgf}), completes the proof
for $\sigma=231$.

As in the proof of Theorem \ref{t-Fine1}, the only difference when
$\sigma=132$ is that the letters of $\rho$ are now smaller than
the letters of $\tau$, but we have the same functional equation and
therefore the same generating function as above.
\end{proof}

We use generating functions to prove the next two results as well.
\begin{thm}
\label{t-123+132}For all $n\geq0$, we have ${\displaystyle d_{n}(123,132)=\frac{2^{n+1}+(7-3n)(-1)^{n}}{9}}.$
\end{thm}

\begin{proof}
For convenience, let $D\coloneqq D(x;123,132)$ and $S\coloneqq S(x;123,132)$.
We first claim that 
\begin{equation}
D=1+x(S-D)+\frac{x^{3}S}{1-x^{2}}.\label{e-123+132fe}
\end{equation}
To see this, first observe that the empty permutation in $\mathfrak{D}_{n}(123,132)$
contributes 1 to the right-hand side of (\ref{e-123+132fe}). If $\pi\in\mathfrak{D}_{n}(123,132)$
is nonempty, then $\pi=\tau n\rho$ where $\tau$ is a decreasing
sequence of even length and $\rho\in\mathfrak{S}_{k}(123,132)$ for
some $0\leq k\leq n-1$. The case where $\tau$ is empty (and $\pi$
is not) gives the contribution $x(S-D)$, as $\rho$ would need to
be a non-desarrangement. Lastly, if $\tau$ is nonempty, then $\rho$
may be either a desarrangement or a non-desarrangement, so this case
contributes $x^{3}S/(1-x^{2})$ to (\ref{e-123+132fe}).

Since $s_{n}(123,132)=2^{n-1}$ for all $n\geq1$ \cite[Proposition 7]{Simion1985},
we have
\[
S=1+\frac{x}{1-2x}.
\]
Substituting this equation into (\ref{e-123+132fe}) and solving for
$D$ yields 
\[
D=\frac{1-2x^{2}}{(1-2x)(1+x)^{2}}.
\]
The sequence with this generating function is given by OEIS entry
A113954 \cite{oeis}, which provides the desired formula for $d_{n}(123,132)$.
\end{proof}

\begin{thm}
\label{t-123+231}For all $n\geq0$, we have ${\displaystyle d_{n}(123,231)=\frac{2n(n-1)+(5-2n)(-1)^{n}+3}{8}}.$
\end{thm}

\begin{proof}
First, the generating functions $D\coloneqq D(x;123,231)$ and $S\coloneqq S(x;123,231)$
satisfy the equation 
\begin{equation}
D=1+x(S-D)+\frac{x^{3}}{(1-x^{2})(1-x)}.\label{e-123+231fe}
\end{equation}
The reasoning is similar to that of (\ref{e-123+132fe}), with the
only difference being that if $\pi\in\mathfrak{D}_{n}(123,231)$ does
not begin with its largest letter $n$, the letters appearing after
$n$ must form a decreasing sequence in order for $\pi$ to avoid
123. Therefore, the term $x^{3}S/(1-x^{2})$ in (\ref{e-123+132fe})
is replaced by $x^{3}/(1-x^{2})(1-x)$.

Next, we have $s_{n}(123,231)={n \choose 2}+1$ for all $n\geq0$
\cite[Lemma 5\ (e) and Proposition 11]{Simion1985}, which has the
generating function 
\begin{equation}
S=\frac{1-2x+2x^{2}}{(1-x)^{3}}\label{e-123+231gf}
\end{equation}
\cite[A152947]{oeis}. Substituting (\ref{e-123+231gf}) into (\ref{e-123+231fe})
and solving for $D$ yields
\[
D=\frac{1-x-x^{2}+3x^{3}}{(1-x)^{2}(1+x)(1-x^{2})},
\]
whose coefficients are given by the desired formula \cite[A130404]{oeis}.
\end{proof}

\begin{thm}
\label{t-123+312}For all $n\geq0$, we have $d_{n}(123,312)={\displaystyle \left\lceil \frac{(n-1)^{2}}{4}\right\rceil }$.
\end{thm}

\begin{proof}
The case $n=0$ is trivial; we assume $n\geq1$ for our proof. An
\textit{interval} of a permutation $\pi$ is a consecutive subsequence
of $\pi$ consisting of consecutive integer letters. For example,
$132$ and $54$ are two of the intervals of $6132754$. Notice that
any nonempty $\pi\in\mathfrak{D}_{n}(123,312)$ can be uniquely decomposed
as $\pi=\tau n\rho$, where $\tau$ is a decreasing interval of even
length and $\rho$ is decreasing, and $\tau$ can be empty if and
only if $n$ is even. Since the choice of $\tau$ completely determines
the permutation $\pi$, it suffices to count the choices for $\tau$.

When $n$ is odd, the length of $\tau$ can be any positive even number
$2k$ up to $n-1$. Assuming $\left|\tau\right|=2k$, there are $n-2k$
choices for the first letter $\tau_{1}$ of $\tau$. Since $\tau$
is a decreasing interval, the choice of $\tau_{1}$ completely determines
$\tau$, so there are there are $\sum_{k=1}^{\left\lfloor (n-1)/2\right\rfloor }(n-2k)$
choices for $\tau$. The same is true when $n$ is even, except that
we may also take $\tau$ to be empty. Hence, we have 
\[
d_{n}(123,312)=\begin{cases}
{\displaystyle \sum_{k=1}^{\left\lfloor (n-1)/2\right\rfloor }(n-2k)},\vphantom{{\displaystyle \frac{\frac{dy_{a}}{dx_{a_{asdf}}}}{\frac{dy_{a_{asdf}}}{dx_{a}}}}} & \text{if }n\text{ is odd,}\\
{\displaystyle 1+\sum_{k=1}^{\left\lfloor (n-1)/2\right\rfloor }(n-2k)},\vphantom{{\displaystyle \frac{\frac{dy_{a}}{dx_{a_{asdf}}}}{\frac{dy_{a_{asdf}}}{dx_{a}}}}} & \text{if }n\text{ is even,}
\end{cases}
\]
and a routine argument shows that this is equal to $\left\lceil (n-1)^{2}/4\right\rceil $
for all $n\geq0$.
\end{proof}

\subsection{Triple restrictions}

Finally, we proceed to the desarrangement avoidance classes $\mathfrak{D}_{n}(\Pi)$
where $\Pi$ contains three patterns of length 3. These are the last
nontrivial desarrangement avoidance classes; all $\mathfrak{D}_{n}(\Pi)$
with $\Pi\subseteq\mathfrak{S}_{3}$ and $\left|\Pi\right|\geq4$
are covered by Propositions \ref{p-trivf}\textendash \ref{p-trivl}.
\begin{thm}
\label{t-123+132+231}For all $n\geq1$, we have $d_{n}(123,132,231)=\begin{cases}
n-1, & \text{if }n\text{ is odd,}\\
1, & \text{if }n\text{ is even.}
\end{cases}$
\end{thm}

\begin{proof}
This is essentially the same as the proof of Proposition \ref{p-trivl}.
The only difference is that, when $n$ is odd, $\pi_{n}$ can be any
of the letters $1,2,\dots,n-1$, since we are no longer avoiding 213
or 312.
\end{proof}

Next, we have five desarrangement avoidance classes whose enumeration
is given by $\left\lfloor n/2\right\rfloor $.
\begin{thm}
\label{t-123+312+friends}For all $n\geq1$ and $\sigma\in\{132,231\}$,
we have $d_{n}(123,312,\sigma)={\displaystyle \left\lfloor \frac{n}{2}\right\rfloor }$.
\end{thm}

\begin{proof}
Take $\sigma=132$. Observe that any $\pi\in\mathfrak{D}_{n}(123,132,312)$
is of the form $\pi=\tau n\rho$ where $\tau=(n-1)(n-2)\cdots(n-k)$
and $\rho=(n-k-1)\cdots21$ for some even $k\in[n]$. Thus, $\pi$
is completely determined by the choice of $k$. Since there are $\left\lfloor n/2\right\rfloor $
choices for $k$, the result follows for $\sigma=132$. The case of
$\sigma=231$ is similar; the only difference is that the letters
of $\tau$ are smaller than the letters of $\rho$.
\end{proof}

\begin{thm}
\label{t-213+231+friends}For all $n\geq1$ and $\sigma\in\{123,132\}$,
we have $d_{n}(213,231,\sigma)={\displaystyle \left\lfloor \frac{n}{2}\right\rfloor }$.
\end{thm}

\begin{proof}
We proceed via induction. The base case ($n=1$) is trivial, so let
us assume that $d_{n}(213,231,\sigma)=\left\lfloor n/2\right\rfloor $
for some $n\geq1$. By Lemma \ref{l-desav}\ (a), we can write each
$\pi\in\mathfrak{D}_{n+1}(213,231,\sigma)$ in the form $\pi=(n+1)\tau$
where $\tau$ is a non-desarrangement in $\mathfrak{S}_{n}(213,231,\sigma)$.
It has already been established (for both $\sigma=123$ and $\sigma=132$)
that $s_{n}(213,231,\sigma)=n$ \cite[Section 4]{Simion1985}, so
\begin{align*}
d_{n+1}(213,231,\sigma) & =n-d_{n}(213,231,\sigma)=n-{\displaystyle \left\lfloor \frac{n}{2}\right\rfloor }={\displaystyle \left\lfloor \frac{n+1}{2}\right\rfloor }
\end{align*}
as desired.
\end{proof}

\begin{thm}
\label{t-132+231+312}For all $n\geq1$, we have $d_{n}(132,231,312)={\displaystyle \left\lfloor \frac{n}{2}\right\rfloor }$.
\end{thm}

\begin{proof}
Again, we use induction, with the base case ($n=1$) being trivial.
Assume $d_{n-1}(132,231,312)=\left\lfloor (n-1)/2\right\rfloor $
for some $n\geq2$. Any $\pi\in\mathfrak{S}_{n}$ avoiding both 132
and 231 must either begin or end with $n$, and if $\pi$ avoids 312
as well, then $\pi$ cannot begin with $n$ unless $\pi$ is the decreasing
permutation $\pi=n\cdots21$. Thus, if $\pi\in\mathfrak{D}_{n}(132,231,312)$,
then either $\pi=\tau n$ where $\tau\in\mathfrak{D}_{n-1}(132,231,312)$,
or $\pi=n\cdots21$ (which can only be the case if $n$ is even).
It follows that
\begin{align*}
d_{n}(132,231,312) & =\begin{cases}
{\displaystyle \left\lfloor \frac{n-1}{2}\right\rfloor },\vphantom{{\displaystyle \frac{\frac{dy_{a}}{dx_{a}}}{\frac{dy_{a}}{dx_{a}}}}} & \text{if }n\text{ is odd,}\\
{\displaystyle 1+\left\lfloor \frac{n-1}{2}\right\rfloor },\vphantom{{\displaystyle \frac{\frac{dy_{a}}{dx_{a}}}{\frac{dy_{a}}{dx_{a}}}}} & \text{if }n\text{ is even,}
\end{cases}\\
 & =\left\lfloor \frac{n}{2}\right\rfloor ,
\end{align*}
which completes the proof.
\end{proof}

Lastly, we have two desarrangement avoidance classes which are enumerated
by the \textit{Fibonacci numbers} $f_{n}$, defined by $f_{n}=f_{n-1}+f_{n-2}$
for all $n\geq2$ and with initial terms $f_{0}=0$ and $f_{1}=1$
\cite[A000045]{oeis}. Our proofs of both results are combinatorial.
\begin{thm}
\label{t-123+132+213}For all $n\geq1$, we have $d_{n}(123,132,213)=f_{n-1}$.
\end{thm}

\begin{proof}
For $n\geq3$, define 
\[
\varphi_{123,132,213}\colon\mathfrak{D}_{n}(123,132,213)\rightarrow\mathfrak{D}_{n-1}(123,132,213)\sqcup\mathfrak{D}_{n-2}(123,132,213)
\]
by 
\[
\varphi_{123,132,213}(\pi)\coloneqq\begin{cases}
\std(\pi_{1}\pi_{2}\cdots\pi_{n-2}), & \text{if }\pi_{n-1}\pi_{n}=21,\\
\std(\pi_{1}\pi_{2}\cdots\pi_{n-1}), & \text{otherwise.}
\end{cases}
\]
For example, we have 
\[
\varphi_{123,132,213}(645321)=4231,\;\varphi_{123,132,213}(645231)=53412,\;\text{and}\;\varphi_{123,132,213}(645312)=53421.
\]
The permutations $\pi$ and $\varphi_{123,132,213}(\pi)$ have the
same first ascent unless $\pi=n\cdots21$, in which case the first
ascent of $\pi$ is $n$ and that of $\varphi_{123,132,213}(\pi)$
is $n-2$. Regardless, the parity of the first ascent is preserved,
so $\varphi_{123,132,213}(\pi)$ is indeed a desarrangement.

We claim that $\varphi_{123,132,213}$ is a bijection. To define its
inverse, first suppose that $\pi\in\mathfrak{D}_{n-2}(123,132,213)$,
and let $\varphi_{123,132,213}^{-1}(\pi)$ be the result of adding
2 to each letter of $\pi$ and then appending the letters $21$. For
example, we have 
\[
\varphi_{123,132,213}^{-1}(4231)=645321.
\]
Observe that $\varphi_{123,132,213}^{-1}(\pi)$ is a desarrangement:
if the first ascent of $\pi$ is $n-2$ then the first ascent of $\varphi_{123,132,213}^{-1}(\pi)$
is $n$, and otherwise they have the same first ascent. Moreover,
$\varphi_{123,132,213}^{-1}$ does not create an occurrence of any
of the three forbidden patterns, so $\varphi_{123,132,213}^{-1}(\pi)\in\mathfrak{D}_{n}(123,132,213)$
whenever $\pi\in\mathfrak{D}_{n-2}(123,132,213)$.

Now, suppose that $\pi\in\mathfrak{D}_{n-1}(123,132,213)$. Since
$\pi$ avoids 123 and 132, either $\pi_{n-2}=1$ or $\pi_{n-1}=1$.
If $\pi_{n-2}=1$, then let $\varphi_{123,132,213}^{-1}(\pi)$ be
the result of adding 1 to each letter of $\pi$ and then appending
the letter $1$. Otherwise, if $\pi_{n-1}=1$, then $\varphi_{123,132,213}^{-1}(\pi)$
is the result of adding 1 to each letter of $\pi$ that is at least
2 and then appending the letter $2$. For example, we have 
\[
\varphi_{123,132,213}^{-1}(53412)=645231\quad\text{and}\quad\varphi_{123,132,213}^{-1}(53421)=645312.
\]
It is readily seen that $\pi$ and $\varphi_{123,132,213}^{-1}(\pi)$
have the same first ascent, and that $\varphi_{123,132,213}^{-1}(\pi)$
avoids the three forbidden patterns, so $\varphi_{123,132,213}^{-1}(\pi)\in\mathfrak{D}_{n}(123,132,213)$.
Finally, note that $\varphi_{123,132,213}^{-1}(\pi)$ does not end
with the letters $21$ when $\pi\in\mathfrak{D}_{n-1}(123,132,213)$,
so the maps $\varphi_{123,132,213}$ and $\varphi_{123,132,213}^{-1}$
are indeed inverses.

Since $\varphi_{123,132,213}$ is a bijection, it follows that $d_{n}(123,132,213)$
satisfies the Fibonacci recurrence, and since $d_{1}(123,132,213)=0=f_{0}$
and $d_{2}(123,132,213)=1=f_{1}$, we conclude that $d_{n}(123,132,213)=f_{n-1}$
for all $n\geq1$.
\end{proof}

\begin{thm}
\label{t-231+312+321}For all $n\geq1$, we have $d_{n}(231,312,321)=f_{n-1}$.
\end{thm}

\begin{proof}
Since $d_{1}(231,312,321)=0=f_{0}$ and $d_{2}(231,312,321)=1=f_{1}$,
it suffices to show that $d_{n}(231,312,321)$ satisfies the Fibonacci
recurrence. As in the proof of Theorem \ref{t-123+132+213}, this
shall be accomplished by constructing a bijection 
\[
\varphi_{231,312,321}\colon\mathfrak{D}_{n}(231,312,321)\rightarrow\mathfrak{D}_{n-1}(231,312,321)\sqcup\mathfrak{D}_{n-2}(231,312,321).
\]

Let $n\geq3$. Note that any permutation $\pi\in\mathfrak{S}_{n}(312,321)$
has either $\pi_{n}=n$ or $\pi_{n-1}=n$. If $\pi$ also avoids 231,
then in the case $\pi_{n-1}=n$, we have $\pi_{n}=n-1$. This allows
us to define $\varphi_{231,312,321}$ by
\[
\varphi_{231,312,321}(\pi)\coloneqq\begin{cases}
\pi_{1}\pi_{2}\cdots\pi_{n-2}, & \text{if }\pi_{n-1}\pi_{n}=n(n-1),\\
\pi_{1}\pi_{2}\cdots\pi_{n-2}\pi_{n-1}, & \text{if }\pi_{n}=n.
\end{cases}
\]
Since $\varphi_{231,312,321}$ preserves the first ascent, it follows that $\varphi_{231,312,321}(\pi)$
is a desarrangement for all $\pi\in\mathfrak{D}_{n}(231,312,321)$.

The inverse of $\varphi_{231,312,321}$ is given by 
\[
\varphi_{231,312,321}^{-1}(\pi)=\begin{cases}
\pi n(n-1), & \text{if }\pi\in\mathfrak{D}_{n-2}(231,312,321),\\
\pi n, & \text{if }\pi\in\mathfrak{D}_{n-1}(231,312,321).
\end{cases}
\]
It is readily seen that $\varphi_{231,312,321}^{-1}$ does not create
any occurrences of the three forbidden patterns, and that $\pi$ and
$\varphi_{231,312,321}^{-1}(\pi)$ have the same first ascent, so
$\varphi_{231,312,321}^{-1}(\pi)\in\mathfrak{D}_{n}(231,312,321)$
for all $\pi\in\mathfrak{D}_{n-1}(231,312,321)\sqcup\mathfrak{D}_{n-2}(231,312,321)$.
The maps $\varphi_{231,312,321}$ and $\varphi_{231,312,321}^{-1}$
are inverses by construction, and we are done.
\end{proof}

\subsection{On pattern-avoiding derangements}

Recall that Robertson\textendash Saracino\textendash Zeilberger \cite{Robertson2002}
and Mansour\textendash Robertson \cite{Mansour2002} have studied
the distribution of $\fix$\textemdash the number of fixed points\textemdash over
all avoidance classes $\mathfrak{S}_{n}(\Pi)$ where $\Pi\subseteq\mathfrak{S}_{3}$
and $\left|\Pi\right|\leq3$. Like our work in Section \ref{s-pa},
Robertson\textendash Saracino\textendash Zeilberger and Mansour\textendash Robertson
achieve their results by way of combinatorial and generating function
techniques. Moreover, these results specialize to results about pattern-avoiding
derangements, and we obtain the following theorem by comparing our
results to theirs. Below, $\tilde{d}_{n}(\Pi)$ denotes the number
of derangements in $\mathfrak{S}_{n}(\Pi)$.
\begin{thm}
\label{t-derpa}Let $\Pi\subseteq\mathfrak{S}_{3}$ with $1\leq\left|\Pi\right|\leq3$.
Then $d_{n}(\Pi)=\tilde{d}_{n}(\Pi)$ for all $n\geq0$ if and only
if $\Pi$ is one of the following: $\{132\}$, $\{132,312\}$, $\{132,321\}$,
$\{213,231\}$, $\{123,132,312\}$, $\{123,213,231\}$, $\{123,312,321\}$,
$\{132,312,321\}$, $\{213,231,312\}$, or $\{213,231,321\}$.
\end{thm}

\noindent The equinumerosities in Theorem \ref{t-derpa}\textemdash at
least the nontrivial ones\textemdash cry out for bijective explanations.
We leave this as an open problem to the interested reader.

We have not attempted to refine our results by the $\pix$ statistic,
but it would be interesting to do so. Computational evidence suggests
that the $\pix$ and $\fix$ statistics are equidistributed over $\mathfrak{S}_{n}(\Pi)$
for all but one $\Pi$ from Theorem \ref{t-derpa}, the sole exception
being $\Pi=\{132\}$. We leave the resolution of this conjecture as
an open problem as well.
\begin{conjecture}
Let $\Pi\subseteq\mathfrak{S}_{3}$ with $1\leq\left|\Pi\right|\leq3$.
Then $\pix$ and $\fix$ are equidistributed over $\mathfrak{S}_{n}(\Pi)$
for all $n\geq0$ if and only if $\Pi$ is one of the following: $\{132,312\}$,
$\{132,321\}$, $\{213,231\}$, $\{123,132,312\}$, $\{123,213,231\}$,
$\{123,312,321\}$, $\{132,312,321\}$, $\{213,231,312\}$, or $\{213,231,321\}$.
\end{conjecture}

\acknowledgements
We are grateful to an anonymous referee for their careful reading of previous versions of this manuscript and providing extensive feedback which led to significant improvements in the exposition. The authors were partially supported by NSF grant DMS-2316181.

\bibliographystyle{plain}
\bibliography{bibliography}

\end{document}